\documentclass[oneside, 10pt]{amsart}
\usepackage{amsaddr}
\usepackage{amssymb}
\usepackage[OT2, T1]{fontenc}
\usepackage[papersize={158mm, 240mm}, body={27pc, 43pc}, hcentering, headsep=18pt, footskip=36pt, vcentering]{geometry}
\usepackage[hidelinks, backref=page]{hyperref}
\usepackage{lmodern}
\usepackage{tabularray}

\newtheorem{theorem}{Theorem}[section]
\newtheorem{lemma}[theorem]{Lemma}
\newtheorem{proposition}[theorem]{Proposition}

\theoremstyle{definition}
\newtheorem{conjecture}[theorem]{Conjecture}

\newtheorem{example}[theorem]{Example}

\newtheorem{remark}[theorem]{Remark}

\renewcommand{\C}{\mathrm{C}}
\renewcommand{\d}{\operatorname{d}}
\newcommand{\EE}{\mathcal{E}}
\newcommand{\F}{\mathbb{F}}
\newcommand{\I}{\mathrm{I}}
\newcommand{\ns}{\operatorname{ns}}

\newcommand{\Q}{\mathbb{Q}}
\newcommand{\Reg}{\operatorname{Reg}}
\newcommand{\tors}{\operatorname{tors}}
\newcommand{\TT}{\mathcal{T}}
\newcommand{\Z}{\mathbb{Z}}

\DeclareSymbolFont{cyrletters}{OT2}{wncyr}{m}{n}
\DeclareMathSymbol{\Sha}{\mathalpha}{cyrletters}{"58}

\title[Torsion and Tamagawa numbers]{Tamagawa numbers and torsion of elliptic curves over function fields}

\author{David Kurniadi Angdinata}
\address{University of East Anglia, School of Engineering, Mathematics and Physics, Norwich Research Park, Norwich NR4 7TJ, United Kingdom}

\author{Mentzelos Melistas}
\address{University of Twente, Department of Applied Mathematics, Drienerlolaan 5, 7522 NB Enschede, The Netherlands}

\begin{document}
\raggedbottom

\begin{abstract}
We study the divisibilities of Tamagawa numbers $ c(E) $ of elliptic curves $ E $ over global function fields $ K $ in terms of their torsion subgroups $ E(K)_{\tors} $. In particular, for a non-isotrivial elliptic curve $ E / k(t) $, where $ k $ is a finite field of characteristic greater than $ 3 $, we prove that $ |E(k(t))_{\tors}|^2 $ divides $ c(E) $, except possibly in four exceptional torsion families. More specifically, we give a complete characterisation of divisibilities for $ c(E) $ in each torsion family, and provide explicit examples to prove that they are best possible. Over a general global function field, we also prove that a rational point of prime order $ 5 \le N \le 101 $ on $ E / K $ forces $ N^2 $ to divide $ c(E) $, which motivates our result. Finally, we formulate a conjecture on the leading coefficient of the $ L $-function of $ E / K $, motivated by the integrality of Birch--Swinnerton-Dyer quotients.
\end{abstract}

\maketitle

\section{Introduction}

Let $ K $ be the function field of a smooth projective geometrically irreducible curve $ C $ of genus $ g $ over a finite field $ k = \F_q $ of characteristic $ p $. For each place $ v $ of $ K $, let $ K_v $ and $ k_v $ be the completion and residue field at $ v $ respectively. For an elliptic curve $ E / K $, its \emph{local Tamagawa number} $ c_v(E) $ is the finite index of the subgroup of non-singular reduction in $ E(K_v) $ \cite[Corollary VII.6.2]{Sil09}. Equivalently, this is the number of $ k_v $-rational connected components of the special fibre of its N\'eron model $ \EE $ at $ v $ \cite[Corollary IV.9.2]{Sil94}. The \emph{Tamagawa number} is then the product $ c(E) := \prod_v c_v(E) $, which satisfies $ c_v(E) = 1 $ for all but finitely many places $ v $ of $ K $.

\pagebreak

The study of $ c(E) $ is motivated by the Birch--Swinnerton-Dyer (BSD) conjecture. Recall that the Mordell--Weil--Lang--N\'eron theorem asserts that $ E(K) $ is a finitely generated abelian group, so $ E(K) \cong \Z^{r'} \oplus E(K)_{\tors} $ for some integer $ r' \ge 0 $, called the \emph{rank}, and some finite group $ E(K)_{\tors} $, called the \emph{torsion subgroup}. Let $ r $ be the order of vanishing of the \emph{$ L $-function} $ L(E, s) $ at $ s = 1 $, and define its normalised leading coefficient to be
$$ L^*(E) := \left.\dfrac{1}{(\log q)^r \cdot r!} \cdot \dfrac{\d^rL(E, s)}{\d s^r}\right|_{s = 1}. $$
The BSD conjecture then predicts $ r = r' $, and that
$$ L^*(E) = \dfrac{\Reg(E) \cdot |\Sha(E)| \cdot c(E)}{q^{h + g - 1} \cdot |E(K)_{\tors}|^2}, $$
where $ \Sha(E) $ is the \emph{Tate--Shafarevich group} and $ \Reg(E) $ is the \emph{regulator} defined in terms of the canonical N\'eron--Tate height divided by $ \log q $. The \emph{height} $ h $ of $ E $ is the degree of the invertible sheaf $ \omega $ on $ C $, given as the pullback of the sheaf of relative differentials on $ \EE $ by the zero section, which is positive if $ E $ is not \emph{isotrivial}. Refer to Gross's notes \cite[Conjecture 2.10(2)]{Gro11} or Ulmer's paper \cite[Page 1073]{Ulm19} for more details.

This paper studies cancellations between $ c(E) $ and $ |E(K)_{\tors}|^2 $. Lorenzini initiated the analogous problem over number fields, and gave precise cancellations for $ \Q $ \cite{Lor11}. There are related works by Krumm \cite{Kru13}, by Najman \cite{Naj17}, by Barrios and Roy \cite{BR22}, by Trbovi\'c \cite{Trb22}, and by the first \cite{Ang25} and second \cite{Mel22, Mel25} named authors. Over global function fields, Ulmer proved that
$$ \dfrac{\Reg(E) \cdot D(E)}{|E(K)_{\tors}|^2} \in \Z, $$
where $ D(E) $ is the absolute discriminant of the trivial lattice over $ k $ of the minimal proper regular model \cite[Proposition 9.1]{Ulm14}, which was later generalised to Jacobians by Berger et al \cite[Proposition 7.2]{BHPPPSSU20}.

\begin{remark}
\label{rem:ulmer}
The quantity $ D(E) $ is related to $ c(E) $, but the implied equality $ D(E) = c(E) $ \cite[Proposition 7.3]{BHPPPSSU20} and the resulting integrality claims \cite[Equations (3.2.1) and (3.2.2)]{Ulm19} are false in general. For instance, the elliptic curve $ E / \F_3(t) $ given by $ y^2 + xy + (t^3 - t + 1)y = x^3 $ has
$$ \dfrac{\Reg(E) \cdot c(E)}{|E(\F_3(t))_{\tors}|^2} = \tfrac{1}{3}. $$
\end{remark}

The following result concerns $ c(E) $ itself outside four torsion families. In what follows, let $ \C_n $ denote the cyclic group of order $ n \ge 1 $.

\begin{theorem}
\label{thm:rational}
Let $ k $ be a finite field of characteristic $ p \ne 2, 3 $, and let $ E / k(t) $ be a non-isotrivial elliptic curve with $ |E(k(t))_{\tors}| \notin \{2, 3, 4\} $. Then
$$ |E(k(t))_{\tors}|^2 \ \text{divides} \ c(E). $$
\end{theorem}

\pagebreak

Theorem \ref{thm:rational} follows from the following result by inspection.

\begin{theorem}
\label{thm:tamagawa}
Let $ k $ be a finite field of characteristic $ p $, and let $ E / k(t) $ be a non-isotrivial elliptic curve with $ E(k(t))_{\tors} \ne 0 $. Then the tuple $ (E(k(t))_{\tors}, p) $ falls into one of the following families. Furthermore, the integer $ c > 0 $ is the largest integer such that $ c \mid c(E) $ for all elliptic curves $ E / k(t) $ in each family.

\bigskip

\begin{center}
\begin{tblr}{|c|c|c|}
\hline
$ E(k(t))_{\tors} $ & $ p $ & $ c $ \\
\hline
$ \C_2 $ & any & $ 2 $ \\
\hline
$ \C_3 $ & any & $ 3 $ \\
\hline
\SetCell[r=2]{m} $ \C_4 $ & $ \ne 2 $ & $ 2^3 $ \\
\cline{2-3}
& $ 2 $ & $ 2^2 $ \\
\hline
$ \C_5 $ & any & $ 5^2 $ \\
\hline
\SetCell[r=3]{m} $ \C_6 $ & $ \ne 2, 3 $ & $ 6^2 $ \\
\cline{2-3}
& $ 2 $ & $ 2 \cdot 6 $ \\
\cline{2-3}
& $ 3 $ & $ 3 \cdot 6 $ \\
\hline
$ \C_7 $ & any & $ 7^3 $ \\
\hline
\SetCell[r=2]{m} $ \C_8 $ & $ \ne 2 $ & $ 8^3 $ \\
\cline{2-3}
& $ 2 $ & $ 8^2 $ \\
\hline
\SetCell[r=2]{m} $ \C_9 $ & $ \ne 3 $ & $ 3 \cdot 9^3 $ \\
\cline{2-3}
& $ 3 $ & $ 9^3 $ \\
\hline
\SetCell[r=3]{m} $ \C_{10} $ & $ \ne 2, 5 $ & $ 5 \cdot 10^3 $ \\
\cline{2-3}
& $ 2 $ & $ 2 \cdot 10^2 $ \\
\cline{2-3}
& $ 5 $ & $ 5^2 \cdot 10^2 $ \\
\hline
\SetCell[r=3]{m} $ \C_{12} $ & $ \ne 2, 3 $ & $ 12^4 $ \\
\cline{2-3}
& $ 2 $ & $ 2^2 \cdot 12^2 $ \\
\cline{2-3}
& $ 3 $ & $ 3 \cdot 6 \cdot 12^2 $ \\
\hline
\end{tblr}
\quad
\begin{tblr}{|c|c|c|}
\hline
$ E(k(t))_{\tors} $ & $ p $ & $ c $ \\
\hline
$ \C_2 \oplus \C_2 $ & $ \ne 2 $ & $ 2^3 $ \\
\hline
$ \C_3 \oplus \C_3 $ & $ \ne 3 $ & $ 3^4 $ \\
\hline
$ \C_4 \oplus \C_2 $ & $ \ne 2 $ & $ 4^2 \cdot 2^2 $ \\
\hline
\SetCell[r=2]{m} $ \C_6 \oplus \C_2 $ & $ \ne 2, 3 $ & $ 2^3 \cdot 6^3 $ \\
\cline{2-3}
& $ 3 $ & $ 6^3 $ \\
\hline
$ \C_8 \oplus \C_2 $ & $ \ne 2 $ & $ 2^{16} $ \\
\hline
$ \C_4 \oplus \C_4 $ & $ \ne 2 $ & $ 4^6 $ \\
\hline
\SetCell[r=2]{m} $ \C_6 \oplus \C_3 $ & $ \ne 2, 3 $ & $ 3^4 \cdot 6^4 $ \\
\cline{2-3}
& $ 2 $ & $ 6^4 $ \\
\hline
$ \C_5 \oplus \C_5 $ & $ \ne 5 $ & $ 5^{12} $ \\
\hline
$ \C_{11} $ & $ 11 $ & $ 11^5 $ \\
\hline
\SetCell[r=2]{m} $ \C_{14} $ & $ 2 $ & $ 2 \cdot 14^3 $ \\
\cline{2-3}
& $ 7 $ & $ 7^3 \cdot 14^3 $ \\
\hline
\SetCell[r=2]{m} $ \C_{15} $ & $ 3 $ & $ 3 \cdot 15^4 $ \\
\cline{2-3}
& $ 5 $ & $ 5^2 \cdot 15^4 $ \\
\hline
$ \C_{18} $ & $ 2 $ & $ 2^2 \cdot 3 \cdot 18^3 $ \\
\hline
$ \C_{10} \oplus \C_2 $ & $ 5 $ & $ 10^6 $ \\
\hline
$ \C_{10} \oplus \C_5 $ & $ 2 $ & $ 10^{12} $ \\
\hline
$ \C_{12} \oplus \C_2 $ & $ 3 $ & $ 6^2 \cdot 12^4 $ \\
\hline
\end{tblr}
\end{center}
\end{theorem}

Most of this paper will be dedicated to proving the divisibilities $ c \mid c(E) $ in Theorem \ref{thm:tamagawa} via a case-by-case analysis, which spans Sections \ref{sec:odd} to \ref{sec:exceptional}. Each case is followed by an explicit example attaining its corresponding value of $ c > 0 $, which proves sharpness. All computations with Weierstrass models were aided by Magma \cite{BCP97} and SageMath \cite{SageMath}. The completeness of the families $ (E(k(t))_{\tors}, p) $ will be clarified in Section \ref{sec:background}.

\pagebreak

Section \ref{sec:proofs} proves the following independent general result.

\begin{theorem}
\label{thm:general}
Let $ 5 \le N \le 101 $ be a prime, let $ k $ be a finite field of characteristic $ p \ne N $, let $ K $ be the function field of a smooth projective geometrically irreducible curve over $ k $, and let $ E / K $ be a non-isotrivial elliptic curve with a rational point $ P $ of prime order $ N $. Then
$$ N^2 \ \text{divides} \ c(E). $$
\end{theorem}

This follows from analysing the \emph{raw form} polynomials $ F_N(X, Y) \in \Z[X, Y] $ defining the modular curve $ Y_1(N) $ computed by Sutherland \cite{Sut}. Theorem \ref{thm:general} is expected to remain valid for all $ p \ne N \ge 5 $, following Lorenzini's conjecture on $ F_N(X, Y) $ \cite[Section 4.4]{Lor25}, and it may even be possible to prove a divisibility for $ |E(K)_{\tors}|^2 $ analogous to Theorem \ref{thm:rational}.

Assuming $ \Sha(E) $ is finite, the following conjecture is in order.

\begin{conjecture}
\label{conj:integrality}
Let $ k $ be a finite field, let $ K $ be the function field of a smooth projective geometrically irreducible curve of genus $ g $ over $ k $, and let $ E / K $ be a non-isotrivial elliptic curve. Then
$$ \dfrac{q^{h + g - 1}}{\Reg(E)} \cdot L^*(E) \in \Z[\tfrac{1}{6}]. $$
\end{conjecture}

The BSD conjecture is known in many contexts, such as when the surface over $ k $ associated to $ E $ is rational or K3, which is implied by $ h \le 2 $ when $ K = k(t) $, or is dominated by a product of curves \cite[Section 3.3]{Gro11}. When $ K = k(t) $, clearly Conjecture \ref{conj:integrality} follows from Theorem \ref{thm:tamagawa} in these cases.

\subsection*{AI declaration}

OpenAI's GPT-5.6 Sol and GPT-6 Astra assisted with corrections of mistakes in the literature (Remark \ref{rem:ulmer}, Proposition \ref{prop:mcdonald}, and Remark \ref{rem:mcdonald}), autoformalisations of the divisibility arguments (Sections \ref{sec:odd} to \ref{sec:exceptional}) in Lean, and formatting of the final paper. The original methodology, case-by-case computations, and Lean formalisation of Tate's algorithm are done by the authors, and predate the use of large language models.

\subsection*{Acknowledgements}

We thank Richard Griffon for pointing out Ulmer's result \cite{Ulm14}. The first-named author was part of the Scalable Theorem Proving via Mathematical Databases project funded by the AI for Math Fund managed by Renaissance Philanthropy in partnership with XTX Markets.

\section{Background}
\label{sec:background}

Recall that an elliptic curve $ E / K $ is \emph{constant} if it arises as the base change of an elliptic curve over $ k $, and $ E / K $ is \emph{isotrivial} if its base change to a finite extension of $ K $ is constant. Equivalently, $ E / K $ is isotrivial if its $ j $-invariant lies in $ k $ \cite[Lecture 1, Remark 1.1.5]{Ulm11}. The following result over $ K = k(t) $ is crucial to the proof of Theorem \ref{thm:tamagawa}.

\pagebreak

\begin{theorem}[Cox and Parry \cite{CP80}, McDonald \cite{McD18}]
\label{thm:torsion}
Let $ k $ be a finite field of characteristic $ p $, and let $ E / k(t) $ be a non-isotrivial elliptic curve with $ E(k(t))_{\tors} \ne 0 $. Then the tuple $ (E(k(t))_{\tors}, p) $ falls into one of the following families. Furthermore, each family admits an explicit parameterisation.

\bigskip

\begin{center}
\scriptsize
\begin{tblr}{|c|c|c|c|c|}
\hline
$ E(k(t))_{\tors} $ & $ p $ & \SetCell[c=2]{c} parameterisation & & reference \\
\hline
\SetCell[r=2]{m} $ \C_2 $ & $ 2 $ & \SetCell[c=2]{c} $ y^2 + xy = x^3 + Ax^2 + Bx $ & & \cite[Prop A.1.1]{Sil09} \\
\cline{2-5}
& $ \ne 2 $ & \SetCell[c=2]{c} $ y^2 = x^3 + Ax^2 + Bx $ & & \cite[Table 1]{McD18} \\
\hline
$ \C_3 $ & any & \SetCell[c=2]{c} $ y^2 + Axy + By = x^3 $ & & \cite[Table 1]{McD18} \\
\hline
$ \C_2 \oplus \C_2 $ & $ \ne 2 $ & \SetCell[c=2]{c} $ y^2 = x(x - A)(x - B) $ & & \cite[Table 1]{McD18} \\
\hline
$ \C_3 \oplus \C_3 $ & $ \ne 3 $ & \SetCell[c=2]{c} $ y^2 + 3(f + 2)xy + 9(f^2 + f + 1)y = x^3 $ & & \cite[Page 401]{McD18} \\
\hline
$ \C_4 $ & any & $ 0 $ & $ f $ & \cite[Table 2]{McD18} \\
\hline
$ \C_5 $ & any & $ f $ & $ a $ & \cite[Table 2]{McD18} \\
\hline
$ \C_6 $ & any & $ f $ & $ a(f + 1) $ & \cite[Table 2]{McD18} \\
\hline
$ \C_7 $ & any & $ f^2 - f $ & $ af $ & \cite[Table 2]{McD18} \\
\hline
$ \C_8 $ & any & $ \tfrac{(2f - 1)(f - 1)}{f} $ & $ af $ & \cite[Table 2]{McD18} \\
\hline
$ \C_9 $ & any & $ f^2(f - 1) $ & $ a(f^2 - f + 1) $ & \cite[Table 2]{McD18} \\
\hline
$ \C_{10} $ & any & $ -\tfrac{f(f - 1)(2f - 1)}{f^2 - 3f + 1} $ & $ -\tfrac{af^2}{f^2 - 3f + 1} $ & \cite[Table 2]{McD18} \\
\hline
$ \C_{12} $ & any & $ -\tfrac{f(2f - 1)(3f^2 - 3f + 1)}{(f - 1)^3} $ & $ -\tfrac{a(2f^2 - 2f + 1)}{f - 1} $ & \cite[Table 2]{McD18} \\
\hline
$ \C_4 \oplus \C_2 $ & $ \ne 2 $ & $ 0 $ & $ f^2 - \tfrac{1}{16} $ & \cite[Table 3]{McD18} \\
\hline
$ \C_6 \oplus \C_2 $ & $ \ne 2 $ & $ -\tfrac{2f - 10}{f^2 - 9} $ & $ \tfrac{a(f - 1)^2}{f^2 - 9} $ & \cite[Table 3]{McD18} \\
\hline
$ \C_8 \oplus \C_2 $ & $ \ne 2 $ & $ \tfrac{(2f + 1)(8f^2 + 4f + 1)}{2(4f + 1)(8f^2 - 1)f} $ & $ \tfrac{2af(4f + 1)}{8f^2 - 1} $ & \cite[Table 3]{McD18} \\
\hline
$ \C_4 \oplus \C_4 $ & $ \ne 2 $ & $ 0 $ & $ f^4 - \tfrac{1}{16} $ & \cite[Table 4]{McD18} \\
\hline
$ \C_6 \oplus \C_3 $ & $ \ne 3 $ & $ -\tfrac{f(f^2 + f + 1)}{(f - 1)^3} $ & $ -\tfrac{a(4f^2 - 2f + 1)}{(f - 1)^3} $ & \cite[Table 4]{McD18} \\
\hline
$ \C_5 \oplus \C_5 $ & $ \ne 5 $ & $ \tfrac{f^4 + 2f^3 + 4f^2 + 3f + 1}{f^5 - 3f^4 + 4f^3 - 2f^2 + f} $ & $ a $ & \cite[Table 4]{McD18} \\
\hline
$ \C_{11} $ & $ 11 $ & $ \tfrac{(f + 3)(f + 5)^2(f + 9)^2}{3(f + 1)(f + 4)^4} $ & $ \tfrac{a(f + 1)^2(f + 9)}{2(f + 4)^3} $ & \cite[Table 14]{McD18} \\
\hline
\SetCell[r=2]{m} $ \C_{14} $ & $ 2 $ & $ \tfrac{f(f + 1)^3}{f^3 + f + 1} $ & $ \tfrac{a}{f^3 + f + 1} $ & \SetCell[r=2]{m} \cite[Table 14]{McD18} \\
\cline{2-4}
& $ 7 $ & $ \tfrac{(f + 1)(f + 3)^3(f + 4)(f + 6)}{f(f + 2)^2(f + 5)} $ & $ \tfrac{a(f + 1)(f + 5)^3}{4f(f + 2)} $ & \\
\hline
\SetCell[r=2]{m} $ \C_{15} $ & $ 3 $ & $ \tfrac{f^3(f + 1)^2}{(f + 2)^6} $ & $ \tfrac{af(f^4 + 2f^3 + f + 1)}{(f + 2)^5} $ & \SetCell[r=2]{m} \cite[Table 14]{McD18} \\
\cline{2-4}
& $ 5 $ & $ \tfrac{(f + 1)(f + 2)^2(f + 4)^3(f^2 + 2)}{(f + 3)^6(f^2 + 3)} $ & $ \tfrac{af(f + 4)}{(f + 3)^5} $ & \\
\hline
$ \C_{18} $ & $ 2 $ & $ \tfrac{f(f + 1)^3(f^2 + f + 1)}{f^3 + f + 1} $ & $ \tfrac{a(f + 1)^2}{f^3 + f + 1} $ & \cite[Table 14]{McD18} \\
\hline
$ \C_{10} \oplus \C_2 $ & $ 5 $ & $ \tfrac{f(f + 1)(f + 2)^2(f + 3)(f + 4)}{(f^2 + 4f + 1)^2} $ & $ \tfrac{a(f + 1)^2(f + 3)^2}{4(f^2 + 4f + 1)^2} $ & \cite[Table 14]{McD18} \\
\hline
$ \C_{10} \oplus \C_5 $ & $ 2 $ & $ \tfrac{f(f^4 + f + 1)(f^4 + f^3 + 1)}{(f^2 + f + 1)^5} $ & $ \tfrac{af^2(f^4 + f^3 + 1)^2}{(f^2 + f + 1)^5} $ & \cite[Table 14]{McD18} \\
\hline
$ \C_{12} \oplus \C_2 $ & $ 3 $ & $ \tfrac{f(f + 1)(f^2 + 1)(f^2 + 2f + 2)}{(f + 2)^3} $ & $ \tfrac{a(f^2 + f + 2)^2}{(f + 2)(f^2 + 2f + 2)} $ & Proposition \ref{prop:mcdonald} \\
\hline
\end{tblr}
\end{center}
\end{theorem}

\pagebreak

Note that a specialisation of any given parameterisation in Theorem \ref{thm:torsion} does not necessarily have torsion subgroup isomorphic to its corresponding $ E(k(t))_{\tors} $, but every elliptic curve $ E / k(t) $ in characteristic $ p $ with torsion subgroup isomorphic to $ E(k(t))_{\tors} $ must be isomorphic to some specialisation of its corresponding parameterisation in Theorem \ref{thm:torsion}. The first five parameterisations by $ A, B \in k(t) $ or $ f \in k(t) \setminus k $ are standard and self-explanatory. The remaining parameterisations by $ f \in k(t) \setminus k $ only list the coefficients $ a, b \in k(t) $ in the \emph{Tate normal form} \cite[Section 2]{Sut12}
$$ \TT(a, b) : y^2 + (1 - a)xy - by = x^3 - bx^2, $$
which always exists when $ E $ has a point of order at least $ 4 $. All of these parameterisations are results of McDonald's explicit computations \cite[Tables 1--4 and 14]{McD18}, except for an error in his $ (\C_{12} \oplus \C_2, 3) $ family.

\begin{proposition}
\label{prop:mcdonald}
Let $ k $ be a finite field of characteristic $ p $, and let $ E / k(t) $ be a non-isotrivial elliptic curve with $ E(k(t))_{\tors} \cong \C_{12} \oplus \C_2 $. Then $ p = 3 $ and $ E \cong \TT(a, b) $, where
$$ a = \dfrac{f(f + 1)(f^2 + 1)(f^2 + 2f + 2)}{(f + 2)^3}, \qquad b = \dfrac{a(f^2 + f + 2)^2}{(f + 2)(f^2 + 2f + 2)}, $$
for some $ f \in k(t) \setminus k $.
\end{proposition}

\begin{proof}
McDonald's classification argument already forces $ p = 3 $, so it suffices to force a $ 3 $-torsion point condition on an elliptic curve $ E / k(t) $ with $ \C_4 \oplus \C_2 \hookrightarrow E(k(t))_{\tors} $, which can be parameterised by $ \TT(0, f_2^2 - \tfrac{1}{16}) = \TT(0, f_2^2 - 1) $ for some $ f_2 \in k(t) \setminus k $ \cite[Pages 420 to 421]{McD18}. This in turn reduces to the $ j $-invariant condition
$$ \dfrac{(f_2^2 + 1)^6}{f_2^2(f_2 + 1)^4(f_2 + 2)^4} \in k(t)^3, $$
and the Hasse invariant condition $ 2(f_2^2 + 1) \in (k(t)^\times)^2 $ \cite[Lecture 1, Proposition 7.3]{Ulm11}. The former forces $ f_2 = f_1^3 $ for some $ f_1 \in k(t) \setminus k $, while the latter forces $ f_2^2 + u^2 = 2 $ for some $ u \in k(t)^\times $. They define a conic $ f_1^2 + w^2 = 2 $ with $ w := -u / (f_1^2 - 2) $, which admits a rational parameterisation
$$ f_1 = \dfrac{f^2 + f + 2}{f^2 + 1}, \qquad w = \dfrac{2f^2 + f + 1}{f^2 + 1}, $$
with base point $ (1, 1) $. Now $ \TT(0, f_1^6 - 1) $ contains the order $ 12 $ point
$$ (w(w + 2)(w + 1)^3, \ 2w^2(w + 2)(w + 1)^5), $$
which can be translated to $ (0, 0) $ via a suitable change of variables, giving
$$ a = \dfrac{w(2w + 1)}{(w + 1)^3}, \qquad b = \dfrac{(2w + 1)(w^2 + 1)}{(w + 1)^4}. $$
This simplifies to the required expressions.
\end{proof}

\pagebreak

In particular, an elliptic curve $ E / k(t) $ being in the $ (\C_{12} \oplus \C_2, 3) $ family does not imply that $ \zeta_4 \in k $, where $ \zeta_n \in \overline{k} $ is a fixed primitive $ n $-th root of unity whenever $ p \nmid n $. For instance, the elliptic curve $ E / \F_3(t) $ given by
$$ y^2 + xy + \dfrac{t^3(t^2 - 1)^3}{(t^2 + 1)^6}y = x^3 + \dfrac{t^3(t^2 - 1)^3}{(t^2 + 1)^6}x^2 $$
has $ E(\F_3(t))_{\tors} \cong \C_{12} \oplus \C_2 $, but $ \zeta_4 \notin \F_3 $.

\begin{remark}
\label{rem:mcdonald}
In addition to the unnecessary requirement $ \zeta_4 \in k $ in \cite[Pages 420 to 421]{McD18}, which is corrected by Proposition \ref{prop:mcdonald}, there are two others worth mentioning. Firstly, the condition in the $ \C_4 \oplus \C_4 $ row of \cite[Table 4]{McD18} should read $ p \ne 2 $ instead of $ p \ne 4 $. Secondly, \cite[Theorem 1.13]{McD18} omits the families $ (\C_6 \oplus \C_3, 2) $ with $ \zeta_3 \in k $ and $ (\C_6 \oplus \C_2, 3) $, despite them appearing later in \cite[Theorems 4.3 and 4.6]{McD18}.
\end{remark}

Computing $ c(E) $ in the rest of the paper largely boils down to running \emph{Tate's algorithm} in each case. For a complete discretely valued field $ F $ with finite residue field and normalised valuation $ v $, and an elliptic curve $ E / F $ with $ v(\Delta) > 0 $ given by a minimal Weierstrass equation
$$ y^2 + a_1xy + a_3y = x^3 + a_2x^2 + a_4x + a_6, $$
Tate gave an explicit algorithm, depending only on $ v(a_i) $ and certain polynomials, to compute various arithmetic data, including Kodaira--N\'eron reduction symbols and $ c_v(E) $. Refer to Silverman's account \cite[Section IV.9]{Sil94} of Tate's original article \cite{Tat75} for more details, but the consequence that will be used most often is summarised as follows. Note that Szydlo's thesis clarifies that the residue field need not be perfect \cite[Theorem 7.1]{Szy99}.

\begin{lemma}
\label{lem:odd}
Let $ F $ be a complete discretely valued field with normalised valuation $ v $, and let $ E / F $ be an elliptic curve given by an integral Weierstrass equation with $ v(a_1) = 0 $ and $ v(a_2), v(a_3), v(a_4), v(a_6) > 0 $. Then the equation is minimal with split multiplicative reduction of type $ \I_{v(\Delta)} $, and $ c_v(E) = v(\Delta) $.
\end{lemma}

\begin{proof}
The valuation conditions give $ v(c_4) = 2v(b_2) = 0 $, so the equation is minimal and Tate's algorithm terminates at step $ 2 $. Since the polynomial $ T^2 + a_1T - a_2 $ splits to $ T(T + a_1) $ in the residue field, the reduction is split multiplicative of type $ \I_{v(\Delta)} $ and $ c_v(E) = v(\Delta) $ \cite[Page 366]{Sil94}.
\end{proof}

The next four sections will be dedicated to the individual cases in Theorem \ref{thm:tamagawa} via repeated applications of Lemma \ref{lem:odd}, as well as other lemmas introduced in their respective sections. They will involve repeatedly computing valuations of Weierstrass equation invariants using standard properties of valuations, as well as some additional facts recorded here, which will be used implicitly without mention or cited explicitly.

\pagebreak

Let $ f \in K \setminus k $. Since $ C $ is smooth, projective, and geometrically irreducible, any global regular function on $ C $ lies in $ k $ \cite[Lemma 33.9.3(8)]{Stacks}, so there is always a place $ v $ of $ K $ with $ v(f) < 0 $. Its valuation on rational expressions of $ f $ can be computed as follows.

\begin{lemma}
\label{lem:negative}
Let $ F $ be a field with a valuation $ v $, let $ f \in F $ be such that $ v(f) < 0 $, and let $ G, H \in F[X] $ be such that their coefficients $ G_i, H_j \in F $ satisfy $ v(G_i), v(H_j) \ge 0 $ and $ v(G_m) = v(H_n) = 0 $, where $ m := \deg G $ and $ n := \deg H $. Then
$$ v(G(f) / H(f)) = v(f)(m - n). $$
In particular, if $ v $ is trivial on $ k \subseteq F $, then $ G(f) \notin k $ for any $ G \in k[X] \setminus k $.
\end{lemma}

\begin{proof}
For any $ 0 \le i < m $ with $ G_i \ne 0 $ and any $ 0 \le j < n $ with $ H_j \ne 0 $,
$$
\begin{aligned}
v(G_if^i) & \ge iv(f) > mv(f) = v(G_mf^m), \\
v(H_jf^j) & \ge jv(f) > nv(f) = v(H_nf^n),
\end{aligned}
$$
so $ G_mf^m $ and $ H_nf^n $ are unique terms with smallest valuations, and hence
$$ v(G(f) / H(f)) = v(G_mf^m) - v(H_nf^n) = v(f)(m - n). $$
If $ v $ is trivial on $ k $, then $ v(G(f)) = mv(f) < 0 $, so $ G(f) \notin k $.
\end{proof}

In particular, for any $ G \in k[X] \setminus k $, applying the same argument to $ 1 / G(f) $ with Lemma \ref{lem:negative} gives a place $ v $ with $ v(G(f)) > 0 $. If $ G = X - m $ for some $ m \in k $, then $ v(f - n) = v(m - n) = 0 $ for any $ n \in k \setminus \{m\} $. The following deals with evaluating general expressions at non-linear places.

\begin{lemma}
\label{lem:identity}
Let $ F $ be a field with a valuation $ v $ that is trivial on $ k \subseteq F $, let $ G, H_i, Q, R \in k[X] $ be such that $ G \notin k $ and
$$ \prod_i H_i(X) = Q(X)G(X) + R(X), $$
and let $ f \in F $ be such that $ v(G(f)) > 0 $ and $ v(R(f)) = 0 $. Then $ v(H_i(f)) = 0 $. In particular, if $ n \in k $ is such that $ G(n) \ne 0 $, then $ v(f - n) = 0 $.
\end{lemma}

\begin{proof}
If $ v(f) < 0 $, then Lemma \ref{lem:negative} gives $ v(G(f)) = v(f)\deg G < 0 $, which contradicts $ v(G(f)) > 0 $, so $ v(f) \ge 0 $. Then $ v(Q(f)G(f)) > 0 $, so
$$ \sum_i v(H_i(f)) = v(Q(f)G(f) + R(f)) = 0, $$
and hence the first statement follows. Now $ G(X) - G(n) = H(X)(X - n) $ for some $ H \in k[X] $, so $ v(f - n) = 0 $ follows from the first statement.
\end{proof}

In what follows, Lemma \ref{lem:identity} will be applied to certain explicit identities, which always exist because the corresponding expressions $ G(X) $ and $ H_i(X) $ are coprime, but can be verified with a computer algebra system.

\pagebreak

\section{Odd cyclic torsion}
\label{sec:odd}

\begin{proposition}
\label{prop:c3}
Let $ k $ be a finite field, and let $ E / k(t) $ be a non-isotrivial elliptic curve with $ E(k(t))_{\tors} \cong \C_3 $. Then $ 3 $ divides $ c(E) $.
\end{proposition}

\begin{proof}
By Theorem \ref{thm:torsion}, $ E $ is given by the Weierstrass equation
$$ y^2 + Axy + By = x^3, \qquad A, B \in k(t). $$
Then $ A \ne 0 $ and $ B / A^3 \notin k $, since
$$ j = \dfrac{A^3(A^3 - 24B)^3}{B^3(A^3 - 27B)} = \dfrac{(24B / A^3 - 1)^3}{(B / A^3)^3(27B / A^3 - 1)} \notin k, $$
so there is a place $ v $ such that $ v(B) > 3v(A) $. Now the change of variables $ (x, y) \mapsto (x / A^2, y / A^3) $ transforms $ E $ to the Weierstrass equation
$$ y^2 + xy + (B / A^3)y = x^3, $$
with $ \Delta = (B / A^3)^3(1 - 27B / A^3) $. At $ v $,
$$ v(a_1) = 0, \qquad v(a_3) = v(B) - 3v(A) > 0, \qquad v(\Delta) = 3(v(B) - 3v(A)). $$
By Lemma \ref{lem:odd}, $ c_v(E) = 3(v(B) - 3v(A)) $, which gives a factor $ 3 \mid c(E) $.
\end{proof}

\begin{example}
The elliptic curve $ E / \F_3(t) $ given by the parameterisation in Proposition \ref{prop:c3} with $ A = t $ and $ B = t^3 + 2t^2 + 1 $ has $ c(E) = 3 $.
\end{example}

\begin{proposition}
\label{prop:c5}
Let $ k $ be a finite field, and let $ E / k(t) $ be a non-isotrivial elliptic curve with $ E(k(t))_{\tors} \cong \C_5 $. Then $ 5^2 $ divides $ c(E) $.
\end{proposition}

\begin{proof}
By Theorem \ref{thm:torsion}, $ E $ is given by the Weierstrass equation
$$ y^2 - (f - 1)xy - fy = x^3 - fx^2, $$
for some $ f \in k(t) \setminus k $, with
$$ \Delta = f^5f_1, \qquad f_1 := f^2 - 11f - 1. $$
There are distinct places $ v_1 $ and $ v_2 $ such that $ v_1(f) > 0 $ and $ v_2(f) < 0 $. At $ v_1 $,
$$ v_1(a_1) = 0, \qquad v_1(a_2) = v_1(a_3) = v_1(f) > 0, \qquad v_1(\Delta) = 5v_1(f). $$
By Lemma \ref{lem:odd}, $ c_{v_1}(E) = 5v_1(f) $, which gives a factor $ 5 \mid c(E) $. Now the change of variables $ (x, y) \mapsto (x / f^2, y / f^3) $ transforms $ E $ to the Weierstrass equation
$$ y^2 - \dfrac{f - 1}{f}xy - \dfrac{1}{f^2}y = x^3 - \dfrac{1}{f}x^2, $$
with $ \Delta = f_1 / f^7 $. At $ v_2 $, Lemma \ref{lem:negative} gives
$$ v_2(a_1) = 0, \qquad
\begin{aligned}
v_2(a_2) & = -v_2(f) > 0, \\
v_2(a_3) & = -2v_2(f) > 0,
\end{aligned}
\qquad v_2(\Delta) = -5v_2(f). $$
By Lemma \ref{lem:odd}, $ c_{v_2}(E) = -5v_2(f) $, which gives another factor $ 5 \mid c(E) $.
\end{proof}

\pagebreak

\begin{example}
The elliptic curve $ E / \F_2(t) $ given by the parameterisation in Proposition \ref{prop:c5} with $ f = t $ has $ c(E) = 5^2 $.
\end{example}

\begin{proposition}
\label{prop:c7}
Let $ k $ be a finite field, and let $ E / k(t) $ be a non-isotrivial elliptic curve with $ E(k(t))_{\tors} \cong \C_7 $. Then $ 7^3 $ divides $ c(E) $.
\end{proposition}

\begin{proof}
By Theorem \ref{thm:torsion}, $ E $ is given by the Weierstrass equation
$$ y^2 - f_1xy - f^2(f - 1)y = x^3 - f^2(f - 1)x^2, $$
for some $ f \in k(t) \setminus k $, with
$$ \Delta = f^7(f - 1)^7f_2, \qquad f_1 := f^2 - f - 1, \qquad f_2 := f^3 - 8f^2 + 5f + 1. $$
There are distinct places $ v_0, v_1, v_2 $ such that
$$ v_0(f) > 0, \qquad v_1(f - 1) > 0, \qquad v_2(f) < 0. $$
At $ v_0 $ and $ v_1 $,
$$
\begin{aligned}
v_0(a_1) & = 0, & v_0(a_2) & = v_0(a_3) = 2v_0(f) > 0, & v_0(\Delta) & = 7v_0(f), \\
v_1(a_1) & = 0, & v_1(a_2) & = v_1(a_3) = v_1(f - 1) > 0, & v_1(\Delta) & = 7v_1(f - 1).
\end{aligned}
$$
By Lemma \ref{lem:odd}, $ c_{v_0}(E) = 7v_0(f) $ and $ c_{v_1}(E) = 7v_1(f - 1) $, which give two factors $ 7 \mid c(E) $. Now the change of variables $ (x, y) \mapsto (x / f^4, y / f^6) $ transforms $ E $ to the Weierstrass equation
$$ y^2 - \dfrac{f_1}{f^2}xy - \dfrac{f - 1}{f^4}y = x^3 - \dfrac{f - 1}{f^2}x^2, $$
with $ \Delta = (f - 1)^7f_2 / f^{17} $. At $ v_2 $, Lemma \ref{lem:negative} gives
$$ v_2(a_1) = 0, \qquad
\begin{aligned}
v_2(a_2) & = -v_2(f) > 0, \\
v_2(a_3) & = -3v_2(f) > 0,
\end{aligned}
\qquad v_2(\Delta) = -7v_2(f). $$
By Lemma \ref{lem:odd}, $ c_{v_2}(E) = -7v_2(f) $, which gives another factor $ 7 \mid c(E) $.
\end{proof}

\begin{example}
The elliptic curve $ E / \F_2(t) $ given by the parameterisation in Proposition \ref{prop:c7} with $ f = t $ has $ c(E) = 7^3 $.
\end{example}

\begin{proposition}
\label{prop:c9}
Let $ k $ be a finite field of characteristic $ p $, and let $ E / k(t) $ be a non-isotrivial elliptic curve with $ E(k(t))_{\tors} \cong \C_9 $.
\begin{enumerate}
\item If $ p \ne 3 $, then $ 3 \cdot 9^3 $ divides $ c(E) $.
\item If $ p = 3 $, then $ 9^3 $ divides $ c(E) $.
\end{enumerate}
\end{proposition}

\begin{proof}
By Theorem \ref{thm:torsion}, $ E $ is given by the Weierstrass equation
$$ y^2 - f_2xy - f^2(f - 1)f_1y = x^3 - f^2(f - 1)f_1x^2, $$
for some $ f \in k(t) \setminus k $, with
$$ \Delta = f^9(f - 1)^9f_1^3f_3, \qquad
\begin{aligned}
f_1 & := f^2 - f + 1, \\
f_2 & := f^3 - f^2 - 1, \\
f_3 & := f^3 - 6f^2 + 3f + 1.
\end{aligned}
$$

\pagebreak

\noindent There are distinct places $ v_0, v_1, v_2 $ such that
$$ v_0(f) > 0, \qquad v_1(f - 1) > 0, \qquad v_2(f) < 0. $$
At $ v_0 $ and $ v_1 $,
$$
\begin{aligned}
v_0(a_1) & = 0, & v_0(a_2) & = v_0(a_3) = 2v_0(f) > 0, & v_0(\Delta) & = 9v_0(f), \\
v_1(a_1) & = 0, & v_1(a_2) & = v_1(a_3) = v_1(f - 1) > 0, & v_1(\Delta) & = 9v_1(f - 1).
\end{aligned}
$$
By Lemma \ref{lem:odd}, $ c_{v_0}(E) = 9v_0(f) $ and $ c_{v_1}(E) = 9v_1(f - 1) $, which give two factors $ 9 \mid c(E) $. Now the change of variables $ (x, y) \mapsto (x / f^6, y / f^9) $ transforms $ E $ to the Weierstrass equation
$$ y^2 - \dfrac{f_2}{f^3}xy - \dfrac{(f - 1)f_1}{f^7}y = x^3 - \dfrac{(f - 1)f_1}{f^4}x^2, $$
with $ \Delta = (f - 1)^9f_1^3f_3 / f^{27} $. At $ v_2 $, Lemma \ref{lem:negative} gives
$$ v_2(a_1) = 0, \qquad
\begin{aligned}
v_2(a_2) & = -v_2(f) > 0, \\
v_2(a_3) & = -4v_2(f) > 0,
\end{aligned}
\qquad v_2(\Delta) = -9v_2(f). $$
By Lemma \ref{lem:odd}, $ c_{v_2}(E) = -9v_2(f) $, which gives another factor $ 9 \mid c(E) $.

Let $ p \ne 3 $. Since $ f_1 \notin k $ by Lemma \ref{lem:negative}, there is a place $ v_3 $, distinct from $ v_0, v_1, v_2 $, such that $ v_3(f_1) > 0 $. At $ v_3 $, Lemma \ref{lem:identity} \footnote{$ f_2f_3 = (f^4 - 6f^3 + 2f^2 + 5f + 8)f_1 - 9 $} gives
$$ v_3(f) = v_3(f - 1) = v_3(f + 1) = v_3(f_2) = v_3(f_3) = 0, $$
so
$$ v_3(a_1) = v_3(f + 1) = 0, \quad v_3(a_2) = v_3(a_3) = v_3(f_1) > 0, \quad v_3(\Delta) = 3v_3(f_1). $$
By Lemma \ref{lem:odd}, $ c_{v_3}(E) = 3v_3(f_1) $, which gives a factor $ 3 \mid c(E) $.
\end{proof}

\begin{example}
The elliptic curve $ E / \F_2(t) $ given by the parameterisation in Proposition \ref{prop:c9} with $ f = t $ has $ c(E) = 3 \cdot 9^3 $. If $ \F_9 := \F_3[\alpha] / (\alpha^2 + 2\alpha + 2) $, then the elliptic curve $ E / \F_9(t) $ given by the same parameterisation with $ f = (\alpha t^3 + 2t^2 + 1) / (t^3 + t^2 + 2) $ has $ c(E) = 9^3 $.
\end{example}

\begin{proposition}
\label{prop:c11}
Let $ k $ be a finite field of characteristic $ 11 $, and let $ E / k(t) $ be a non-isotrivial elliptic curve with $ E(k(t))_{\tors} \cong \C_{11} $. Then $ 11^5 $ divides $ c(E) $.
\end{proposition}

\begin{proof}
By Theorem \ref{thm:torsion}, $ E $ is given by the Weierstrass equation
$$
\begin{aligned}
& y^2 + \dfrac{8f_1}{(f + 1)(f + 4)^4}xy - \dfrac{(f + 1)(f + 3)(f + 5)^2(f + 9)^3}{6(f + 4)^7}y \\
& \qquad = x^3 - \dfrac{(f + 1)(f + 3)(f + 5)^2(f + 9)^3}{6(f + 4)^7}x^2,
\end{aligned}
$$

\pagebreak

\noindent for some $ f \in k(t) \setminus k $, with
$$ \Delta = \dfrac{2f^2(f + 3)^{11}(f + 5)^{11}(f + 9)^{11}}{(f + 1)(f + 4)^{37}}, \quad f_1 := f^5 + 10f^4 + 5f^3 + 8f^2 + 5f + 3. $$
There are distinct places $ v_1, \dots, v_5 $ such that
$$ v_1(f + 1) > 0, \ v_2(f + 3) > 0, \ v_3(f + 4) > 0, \ v_4(f + 5) > 0, \ v_5(f + 9) > 0. $$
At $ v_2, v_4, v_5 $,
$$
\begin{aligned}
v_2(a_1) & = 0, & v_2(a_2) & = v_2(a_3) = v_2(f + 3) > 0, & v_2(\Delta) & = 11v_2(f + 3), \\
v_4(a_1) & = 0, & v_4(a_2) & = v_4(a_3) = 2v_4(f + 5) > 0, & v_4(\Delta) & = 11v_4(f + 5), \\
v_5(a_1) & = 0, & v_5(a_2) & = v_5(a_3) = 3v_5(f + 9) > 0, & v_5(\Delta) & = 11v_5(f + 9).
\end{aligned}
$$
By Lemma \ref{lem:odd},
$$ c_{v_2}(E) = 11v_2(f + 3), \qquad c_{v_4}(E) = 11v_4(f + 5), \qquad c_{v_5}(E) = 11v_5(f + 9), $$
which give three factors $ 11 \mid c(E) $. Now the change of variables
$$ (x, y) \mapsto ((f + 1)^2(f + 4)^8x, (f + 1)^3(f + 4)^{12}y) $$
transforms $ E $ to the Weierstrass equation
$$
\begin{aligned}
& y^2 + 8f_1xy - \tfrac{1}{6}(f + 1)^4(f + 3)(f + 4)^5(f + 5)^2(f + 9)^3y \\
& \qquad = x^3 - \tfrac{1}{6}(f + 1)^3(f + 3)(f + 4)(f + 5)^2(f + 9)^3x^2,
\end{aligned}
$$
with
$$ \Delta = 2f^2(f + 1)^{11}(f + 3)^{11}(f + 4)^{11}(f + 5)^{11}(f + 9)^{11}. $$
At $ v_1 $ and $ v_3 $,
$$
\begin{aligned}
v_1(a_1) & = 0, &
\begin{aligned}
v_1(a_2) & = 3v_1(f + 1) > 0, \\
v_1(a_3) & = 4v_1(f + 1) > 0,
\end{aligned}
& & v_1(\Delta) & = 11v_1(f + 1), \\
v_3(a_1) & = 0, &
\begin{aligned}
v_3(a_2) & = v_3(f + 4) > 0, \\
v_3(a_3) & = 5v_3(f + 4) > 0,
\end{aligned}
& & v_3(\Delta) & = 11v_3(f + 4).
\end{aligned}
$$
By Lemma \ref{lem:odd}, $ c_{v_1}(E) = 11v_1(f + 1) $ and $ c_{v_3}(E) = 11v_3(f + 4) $, which give two more factors $ 11 \mid c(E) $.
\end{proof}

\begin{example}
If $ \F_{121} := \F_{11}[\alpha] / (\alpha^2 + 7\alpha + 2) $, then the elliptic curve $ E / \F_{121}(t) $ given by the parameterisation in Proposition \ref{prop:c11} with $ f = \alpha t^4 $ has $ c(E) = 11^5 $.
\end{example}

\begin{proposition}
\label{prop:c15}
Let $ k $ be a finite field of characteristic $ p $, and let $ E / k(t) $ be a non-isotrivial elliptic curve with $ E(k(t))_{\tors} \cong \C_{15} $.
\begin{enumerate}
\item If $ p = 3 $, then $ 3 \cdot 15^4 $ divides $ c(E) $.
\item If $ p = 5 $, then $ 5^2 \cdot 15^4 $ divides $ c(E) $.
\end{enumerate}
\end{proposition}

\pagebreak

\begin{proof}
Let $ p = 3 $. By Theorem \ref{thm:torsion}, $ E $ is given by the Weierstrass equation
$$ y^2 + \dfrac{f_2}{(f + 2)^6}xy - \dfrac{f^4(f + 1)^2f_1}{(f + 2)^{11}}y = x^3 - \dfrac{f^4(f + 1)^2f_1}{(f + 2)^{11}}x^2, $$
for some $ f \in k(t) \setminus k $, with
$$ \Delta = \dfrac{2f^{15}(f + 1)^{15}f_1^3}{(f + 2)^{57}}, \quad f_1 := f^4 + 2f^3 + f + 1, \quad f_2 := f^6 + 2f^5 + f^4 + 1. $$
There are distinct places $ v_0, \dots, v_3 $ such that
$$ v_0(f) > 0, \qquad v_1(f + 1) > 0, \qquad v_2(f + 2) > 0, \qquad v_3(f) < 0. $$
At $ v_0 $ and $ v_1 $,
$$
\begin{aligned}
v_0(a_1) & = 0, & v_0(a_2) & = v_0(a_3) = 4v_0(f) > 0, & v_0(\Delta) & = 15v_0(f), \\
v_1(a_1) & = 0, & v_1(a_2) & = v_1(a_3) = 2v_1(f + 1) > 0, & v_1(\Delta) & = 15v_1(f + 1).
\end{aligned}
$$
At $ v_3 $, Lemma \ref{lem:negative} gives
$$ v_3(a_1) = 0, \qquad v_3(a_2) = v_3(a_3) = -v_3(f) > 0, \qquad v_3(\Delta) = -15v_3(f). $$
Since $ f_1 \notin k $ by Lemma \ref{lem:negative}, there is a place $ v_4 $, distinct from $ v_0, \dots, v_3 $, such that $ v_4(f_1) > 0 $. At $ v_4 $, Lemma \ref{lem:identity} \footnote{$ f_2 = (f^2 + 1)f_1 - f(f + 1) $} gives
$$ v_4(f) = v_4(f + 1) = v_4(f + 2) = v_4(f_2) = 0, $$
so
$$ v_4(a_1) = 0, \qquad v_4(a_2) = v_4(a_3) = v_4(f_1) > 0, \qquad v_4(\Delta) = 3v_4(f_1). $$
By Lemma \ref{lem:odd},
$$
\begin{aligned}
c_{v_0}(E) & = 15v_0(f), & c_{v_1}(E) & = 15v_1(f + 1), \\
c_{v_3}(E) & = -15v_3(f), & c_{v_4}(E) & = 3v_4(f_1),
\end{aligned}
$$
which give three factors $ 15 \mid c(E) $ and a factor $ 3 \mid c(E) $. Now the change of variables $ (x, y) \mapsto ((f + 2)^{12}x, (f + 2)^{18}y) $ transforms $ E $ to the Weierstrass equation
$$ y^2 + f_2xy - f^4(f + 1)^2(f + 2)^7f_1y = x^3 - f^4(f + 1)^2(f + 2)f_1x^2, $$
with
$$ \Delta = 2f^{15}(f + 1)^{15}(f + 2)^{15}f_1^3. $$
At $ v_2 $,
$$ v_2(a_1) = 0, \qquad
\begin{aligned}
v_2(a_2) & = v_2(f + 2) > 0, \\
v_2(a_3) & = 7v_2(f + 2) > 0,
\end{aligned}
\qquad v_2(\Delta) = 15v_2(f + 2). $$
By Lemma \ref{lem:odd}, $ c_{v_2}(E) = 15v_2(f + 2) $, which gives another factor $ 15 \mid c(E) $.

\pagebreak

\noindent Now let $ p = 5 $. By Theorem \ref{thm:torsion}, $ E $ is given by the Weierstrass equation
$$
\begin{aligned}
& y^2 + \dfrac{ff_3f_4}{(f + 3)^6(f^2 + 3)}xy - \dfrac{f(f + 1)(f + 2)^2(f + 4)^4(f^2 + 2)}{(f + 3)^{11}(f^2 + 3)}y \\
& \qquad = x^3 - \dfrac{f(f + 1)(f + 2)^2(f + 4)^4(f^2 + 2)}{(f + 3)^{11}(f^2 + 3)}x^2,
\end{aligned}
$$
for some $ f \in k(t) \setminus k $, with
$$ \Delta = \dfrac{2f^4(f + 1)^{15}(f + 2)^{15}(f + 4)^{15}(f^2 + 2)^5}{(f + 3)^{57}(f^2 + 3)^7}, $$
where $ f_3 := f^2 + f + 2 $ and $ f_4 := f^4 + 4f^3 + 3f + 3 $. There are distinct places $ v_1, \dots, v_4 $ such that
$$ v_1(f + 1) > 0, \qquad v_2(f + 2) > 0, \qquad v_3(f + 3) > 0, \qquad v_4(f + 4) > 0. $$
At $ v_1, v_2, v_4 $,
$$
\begin{aligned}
v_1(a_1) & = 0, & v_1(a_2) & = v_1(a_3) = v_1(f + 1) > 0, & v_1(\Delta) & = 15v_1(f + 1), \\
v_2(a_1) & = 0, & v_2(a_2) & = v_2(a_3) = 2v_2(f + 2) > 0, & v_2(\Delta) & = 15v_2(f + 2), \\
v_4(a_1) & = 0, & v_4(a_2) & = v_4(a_3) = 4v_4(f + 4) > 0, & v_4(\Delta) & = 15v_4(f + 4).
\end{aligned}
$$
Since $ f^2 + 2 \notin k $ by Lemma \ref{lem:negative}, there is a place $ v_5 $, distinct from $ v_1, \dots, v_4 $, such that $ v_5(f^2 + 2) > 0 $. At $ v_5 $, Lemma \ref{lem:identity} \footnote{$ (f^2 + 3)f_3f_4 = (f^6 + 2f^4 + f^3 - 1)(f^2 + 2) + 2f $} gives
$$
\begin{aligned}
& v_5(f) = v_5(f + 1) = v_5(f + 2) = v_5(f + 3) = v_5(f + 4) = 0, \\
& v_5(f^2 + 3) = v_5(f_3) = v_5(f_4) = 0,
\end{aligned}
$$
so
$$ v_5(a_1) = 0, \qquad v_5(a_2) = v_5(a_3) = v_5(f^2 + 2) > 0, \qquad v_5(\Delta) = 5v_5(f^2 + 2). $$
By Lemma \ref{lem:odd},
$$
\begin{aligned}
c_{v_1}(E) & = 15v_1(f + 1), & c_{v_2}(E) & = 15v_2(f + 2), \\
c_{v_4}(E) & = 15v_4(f + 4), & c_{v_5}(E) & = 5v_5(f^2 + 2),
\end{aligned}
$$
which give three factors $ 15 \mid c(E) $ and a factor $ 5 \mid c(E) $. Now the change of variables
$$ (x, y) \mapsto ((f + 3)^{12}(f^2 + 3)^2x, (f + 3)^{18}(f^2 + 3)^3y) $$
transforms $ E $ to the Weierstrass equation
$$
\begin{aligned}
& y^2 + ff_3f_4xy - f(f + 1)(f + 2)^2(f + 3)^7(f + 4)^4(f^2 + 2)(f^2 + 3)^2y \\
& \qquad = x^3 - f(f + 1)(f + 2)^2(f + 3)(f + 4)^4(f^2 + 2)(f^2 + 3)x^2,
\end{aligned}
$$
with
$$ \Delta = 2f^4(f + 1)^{15}(f + 2)^{15}(f + 3)^{15}(f + 4)^{15}(f^2 + 2)^5(f^2 + 3)^5. $$

\pagebreak

\noindent At $ v_3 $,
$$ v_3(a_1) = 0, \qquad
\begin{aligned}
v_3(a_2) & = v_3(f + 3) > 0, \\
v_3(a_3) & = 7v_3(f + 3) > 0,
\end{aligned}
\qquad v_3(\Delta) = 15v_3(f + 3). $$
Since $ f^2 + 3 \notin k $ by Lemma \ref{lem:negative}, there is a place $ v_6 $, distinct from $ v_1, \dots, v_5 $, such that $ v_6(f^2 + 3) > 0 $. At $ v_6 $, Lemma \ref{lem:identity} \footnote{$ (f^2 + 2)f_3f_4 = (f^6 + f^3 - 2f^2 - 2f - 1)(f^2 + 3) - f $} gives
$$
\begin{aligned}
& v_6(f) = v_6(f + 1) = v_6(f + 2) = v_6(f + 3) = v_6(f + 4) = 0, \\
& v_6(f^2 + 2) = v_6(f_3) = v_6(f_4) = 0,
\end{aligned}
$$
so
$$ v_6(a_1) = 0, \qquad
\begin{aligned}
v_6(a_2) & = v_6(f^2 + 3) > 0, \\
v_6(a_3) & = 2v_6(f^2 + 3) > 0,
\end{aligned}
\qquad v_6(\Delta) = 5v_6(f^2 + 3). $$
By Lemma \ref{lem:odd}, $ c_{v_3}(E) = 15v_3(f + 3) $ and $ c_{v_6}(E) = 5v_6(f^2 + 3) $, which give another factor $ 15 \mid c(E) $ and another factor $ 5 \mid c(E) $.
\end{proof}

\begin{example}
The elliptic curve $ E / \F_3(t) $ given by its parameterisation in Proposition \ref{prop:c15} with $ f = t $ has $ c(E) = 3 \cdot 15^4 $. If $ \F_{125} := \F_5[\alpha] / (\alpha^3 + 3\alpha + 3) $, then the elliptic curve $ E / \F_{125}(t) $ given by its parameterisation in Proposition \ref{prop:c15} with
$$ f = \dfrac{t^3}{t^2 + (\alpha^2 + \alpha + 1)t + (\alpha^2 + 3\alpha + 3)} $$
has $ c(E) = 5^2 \cdot 15^4 $.
\end{example}

\section{Even cyclic torsion}
\label{sec:even}

\begin{lemma}
\label{lem:even}
Let $ F $ be a complete discretely valued field with finite residue field and normalised valuation $ v $, and let $ E / F $ be an elliptic curve given by an integral Weierstrass equation with $ v(c_4) = 0 $ and $ v(\Delta) > 0 $ even. Then the equation is minimal with multiplicative reduction of type $ \I_{v(\Delta)} $, and $ c_v(E) $ is even.
\end{lemma}

\begin{proof}
Since $ v(c_4) = 0 $, the equation is minimal and Tate's algorithm terminates at step $ 2 $, so the reduction is multiplicative of type $ \I_{v(\Delta)} $. If the reduction is split, then $ c_v(E) = v(\Delta) $ is even, otherwise $ c_v(E) = 2 $ is also even since $ v(\Delta) $ is even \cite[Page 366]{Sil94}.
\end{proof}

\begin{proposition}
\label{prop:c6}
Let $ k $ be a finite field of characteristic $ p $, and let $ E / k(t) $ be a non-isotrivial elliptic curve with $ E(k(t))_{\tors} \cong \C_6 $.
\begin{enumerate}
\item If $ p \ne 2, 3 $, then $ 6^2 $ divides $ c(E) $.
\item If $ p = 2 $, then $ 2 \cdot 6 $ divides $ c(E) $.
\item If $ p = 3 $, then $ 3 \cdot 6 $ divides $ c(E) $.
\end{enumerate}
\end{proposition}

\pagebreak

\begin{proof}
By Theorem \ref{thm:torsion}, $ E $ is given by the Weierstrass equation
$$ y^2 - (f - 1)xy - f(f + 1)y = x^3 - f(f + 1)x^2, $$
for some $ f \in k(t) \setminus k $, with $ \Delta = f^6(f + 1)^3(9f + 1) $. There are distinct places $ v_0, v_1, v_2 $ such that
$$ v_0(f) > 0, \qquad v_1(f + 1) > 0, \qquad v_2(f) < 0. $$
At $ v_0 $,
$$ v_0(a_1) = 0, \qquad v_0(a_2) = v_0(a_3) = v_0(f) > 0, \qquad v_0(\Delta) = 6v_0(f). $$
By Lemma \ref{lem:odd}, $ c_{v_0}(E) = 6v_0(f) $, which gives a factor $ 6 \mid c(E) $.

Let $ p \ne 2 $. At $ v_1 $,
$$ v_1(a_1) = 0, \qquad v_1(a_2) = v_1(a_3) = v_1(f + 1) > 0, \qquad v_1(\Delta) = 3v_1(f + 1). $$
By Lemma \ref{lem:odd}, $ c_{v_1}(E) = 3v_1(f + 1) $, which gives a factor $ 3 \mid c(E) $.

Now let $ p \ne 3 $. The change of variables $ (x, y) \mapsto (x / f^2, y / f^3) $ transforms $ E $ to the Weierstrass equation
$$ y^2 - \dfrac{f - 1}{f}xy - \dfrac{f + 1}{f^2}y = x^3 - \dfrac{f + 1}{f}x^2, $$
with $ \Delta = (f + 1)^3(9f + 1) / f^6 $. At $ v_2 $, Lemma \ref{lem:negative} gives
$$ v_2(c_4) = 2v_2(b_2) = 2v_2(3) = 0, \qquad v_2(\Delta) = -2v_2(f). $$
By Lemma \ref{lem:even}, $ 2 \mid c_{v_2}(E) $, which gives a factor $ 2 \mid c(E) $.
\end{proof}

\begin{example}
The elliptic curve $ E / \F_5(t) $ given by the parameterisation in Proposition \ref{prop:c6} with $ f = t $ has $ c(E) = 6^2 $. If $ \F_4 := \F_2[\alpha] / (\alpha^2 + \alpha + 1) $, then the elliptic curve $ E / \F_4(t) $ given by the same parameterisation with $ f = (t^3 + \alpha) / (t^3 + \alpha^2) $ has $ c(E) = 2 \cdot 6 $. If $ \F_9 := \F_3[\alpha] / (\alpha^2 + 2\alpha + 2) $, then the elliptic curve $ E / \F_9(t) $ given by the same parameterisation with $ f = \alpha t^4 + 1 $ has $ c(E) = 3 \cdot 6 $.
\end{example}

\begin{proposition}
\label{prop:c8}
Let $ k $ be a finite field of characteristic $ p $, and let $ E / k(t) $ be a non-isotrivial elliptic curve with $ E(k(t))_{\tors} \cong \C_8 $.
\begin{enumerate}
\item If $ p \ne 2 $, then $ 8^3 $ divides $ c(E) $.
\item If $ p = 2 $, then $ 8^2 $ divides $ c(E) $.
\end{enumerate}
\end{proposition}

\begin{proof}
By Theorem \ref{thm:torsion}, $ E $ is given by the Weierstrass equation
$$ y^2 - \dfrac{f_1}{f}xy - (f - 1)(2f - 1)y = x^3 - (f - 1)(2f - 1)x^2, $$
for some $ f \in k(t) \setminus k $, with
$$ \Delta = \dfrac{(f - 1)^8(2f - 1)^4f_2}{f^4}, \qquad f_1 := 2f^2 - 4f + 1, \qquad f_2 := 8f^2 - 8f + 1. $$

\pagebreak

\noindent There are distinct places $ v_0 $ and $ v_1 $ such that $ v_0(f) > 0 $ and $ v_1(f - 1) > 0 $. At $ v_1 $,
$$ v_1(a_1) = 0, \qquad v_1(a_2) = v_1(a_3) = v_1(f - 1) > 0, \qquad v_1(\Delta) = 8v_1(f - 1). $$
By Lemma \ref{lem:odd}, $ c_{v_1}(E) = 8v_1(f - 1) $, which gives a factor $ 8 \mid c(E) $. Now the change of variables
$$ (x, y) \mapsto \left(\dfrac{f^2}{(f - 1)^4}x, \dfrac{f^3}{(f - 1)^6}y\right) $$
transforms $ E $ to the Weierstrass equation
$$ y^2 - \dfrac{f_1}{(f - 1)^2}xy - \dfrac{f^3(2f - 1)}{(f - 1)^5}y = x^3 - \dfrac{f^2(2f - 1)}{(f - 1)^3}x^2, $$
with $ \Delta = f^8(2f - 1)^4f_2 / (f - 1)^{16} $. At $ v_0 $,
$$ v_0(a_1) = 0, \qquad
\begin{aligned}
v_0(a_2) & = 2v_0(f) > 0, \\
v_0(a_3) & = 3v_0(f) > 0,
\end{aligned}
\qquad v_0(\Delta) = 8v_0(f). $$
By Lemma \ref{lem:odd}, $ c_{v_0}(E) = 8v_0(f) $, which gives another factor $ 8 \mid c(E) $.

Let $ p \ne 2 $. There are places $ v_2 $ and $ v_3 $, distinct from $ v_0 $ and $ v_1 $, such that $ v_2(2f - 1) > 0 $ and $ v_3(f) < 0 $. At $ v_2 $,
$$ v_2(a_1) = 0, \qquad v_2(a_2) = v_2(a_3) = v_2(2f - 1) > 0, \qquad v_2(\Delta) = 4v_2(2f - 1). $$
By Lemma \ref{lem:odd}, $ c_{v_2}(E) = 4v_2(2f - 1) $, which gives a factor $ 4 \mid c(E) $. At $ v_3 $, Lemma \ref{lem:negative} gives
$$ v_3(c_4) = 2v_3(b_2) = 2v_3(4) = 0, \qquad v_3(\Delta) = -2v_3(f). $$
By Lemma \ref{lem:even}, $ 2 \mid c_{v_3}(E) $, which gives a factor $ 2 \mid c(E) $.
\end{proof}

\begin{example}
The elliptic curve $ E / \F_3(t) $ given by the parameterisation in Proposition \ref{prop:c8} with $ f = t $ has $ c(E) = 8^3 $. If $ \F_4 := \F_2[\alpha] / (\alpha^2 + \alpha + 1) $, then the elliptic curve $ E / \F_4(t) $ given by the same parameterisation with $ f = t^8 + t^5 + t^3 + \alpha $ has $ c(E) = 8^2 $.
\end{example}

\begin{proposition}
\label{prop:c10}
Let $ k $ be a finite field of characteristic $ p $, and let $ E / k(t) $ be a non-isotrivial elliptic curve with $ E(k(t))_{\tors} \cong \C_{10} $.
\begin{enumerate}
\item If $ p \ne 2, 5 $, then $ 5 \cdot 10^3 $ divides $ c(E) $.
\item If $ p = 2 $, then $ 2 \cdot 10^2 $ divides $ c(E) $.
\item If $ p = 5 $, then $ 5^2 \cdot 10^2 $ divides $ c(E) $.
\end{enumerate}
\end{proposition}

\begin{proof}
By Theorem \ref{thm:torsion}, $ E $ is given by the Weierstrass equation
$$ y^2 + \dfrac{f_3}{f_1}xy - \dfrac{f^3(f - 1)(2f - 1)}{f_1^2}y = x^3 - \dfrac{f^3(f - 1)(2f - 1)}{f_1^2}x^2, $$

\pagebreak

\noindent for some $ f \in k(t) \setminus k $, with
$$ \Delta = \dfrac{f^{10}(f - 1)^{10}(2f - 1)^5f_2}{f_1^{10}}, \qquad
\begin{aligned}
f_1 & := f^2 - 3f + 1, \\
f_2 & := 4f^2 - 2f - 1, \\
f_3 & := 2f^3 - 2f^2 - 2f + 1.
\end{aligned}
$$
There are distinct places $ v_0 $ and $ v_1 $ such that $ v_0(f) > 0 $ and $ v_1(f - 1) > 0 $. At $ v_0 $ and $ v_1 $,
$$
\begin{aligned}
v_0(a_1) & = 0, & v_0(a_2) & = v_0(a_3) = 3v_0(f) > 0, & v_0(\Delta) & = 10v_0(f), \\
v_1(a_1) & = 0, & v_1(a_2) & = v_1(a_3) = v_1(f - 1) > 0, & v_1(\Delta) & = 10v_1(f - 1).
\end{aligned}
$$
By Lemma \ref{lem:odd}, $ c_{v_0}(E) = 10v_0(f) $ and $ c_{v_1}(E) = 10v_1(f - 1) $, which give two factors $ 10 \mid c(E) $. Now the change of variables
$$ (x, y) \mapsto \left(\dfrac{f_1^2}{f^6}x, \dfrac{f_1^3}{f^9}y\right) $$
transforms $ E $ to the Weierstrass equation
$$ y^2 + \dfrac{f_3}{f^3}xy - \dfrac{(f - 1)(2f - 1)f_1}{f^6}y = x^3 - \dfrac{(f - 1)(2f - 1)}{f^3}x^2, $$
with
$$ \Delta = \dfrac{(f - 1)^{10}(2f - 1)^5f_1^2f_2}{f^{26}}. $$

Let $ p \ne 2 $. There are places $ v_2 $ and $ v_3 $, distinct from $ v_0 $ and $ v_1 $, such that $ v_2(2f - 1) > 0 $ and $ v_3(f) < 0 $. At $ v_2 $,
$$ v_2(a_1) = 0, \qquad v_2(a_2) = v_2(a_3) = v_2(2f - 1) > 0, \qquad v_2(\Delta) = 5v_2(2f - 1). $$
At $ v_3 $, Lemma \ref{lem:negative} gives
$$ v_3(a_1) = 0, \qquad
\begin{aligned}
v_3(a_2) & = -v_3(f) > 0, \\
v_3(a_3) & = -2v_3(f) > 0,
\end{aligned}
\qquad v_3(\Delta) = -5v_3(f). $$
By Lemma \ref{lem:odd}, $ c_{v_2}(E) = 5v_2(2f - 1) $ and $ c_{v_3}(E) = -5v_3(f) $, which give two factors $ 5 \mid c(E) $.

Now let $ p \ne 5 $. Since $ f_1 \notin k $ by Lemma \ref{lem:negative}, there is a place $ v_4 $, distinct from $ v_0, \dots, v_3 $, such that $ v_4(f_1) > 0 $. At $ v_4 $, Lemma \ref{lem:identity} \footnote{$ f_2 = 4f_1 + 5(2f - 1) $ and $ f^6b_2 = 4(f^4 - 17f^3 + 26f^2 - 15f + 3)f_1 + (2f - 1)^4(4f - 11) $} gives
$$ v_4(f) = v_4(f - 1) = v_4(2f - 1) = v_4(4f - 11) = v_4(f_2) = v_4(f^6b_2) = 0, $$
so
$$ v_4(c_4) = 2v_4(b_2) = 2v_4(4f - 11) = 0, \qquad v_4(\Delta) = 2v_4(f_1). $$
By Lemma \ref{lem:even}, $ 2 \mid c_{v_4}(E) $, which gives a factor $ 2 \mid c(E) $.
\end{proof}

\pagebreak

\begin{example}
The elliptic curve $ E / \F_3(t) $ given by the parameterisation in Proposition \ref{prop:c10} with $ f = t $ has $ c(E) = 5 \cdot 10^3 $. The elliptic curve $ E / \F_2(t) $ given by the same parameterisation with $ f = t $ has $ c(E) = 2 \cdot 10^2 $. If $ \F_{25} := \F_5[\alpha] / (\alpha^2 + 4\alpha + 2) $, then the elliptic curve $ E / \F_{25}(t) $ given by the same parameterisation with $ f = (2t^2 + \alpha) / (t^2 + 4\alpha) $ has $ c(E) = 5^2 \cdot 10^2 $.
\end{example}

\begin{proposition}
\label{prop:c12}
Let $ k $ be a finite field of characteristic $ p $, and let $ E / k(t) $ be a non-isotrivial elliptic curve with $ E(k(t))_{\tors} \cong \C_{12} $.
\begin{enumerate}
\item If $ p \ne 2, 3 $, then $ 12^4 $ divides $ c(E) $.
\item If $ p = 2 $, then $ 2^2 \cdot 12^2 $ divides $ c(E) $.
\item If $ p = 3 $, then $ 3 \cdot 6 \cdot 12^2 $ divides $ c(E) $.
\end{enumerate}
\end{proposition}

\begin{proof}
By Theorem \ref{thm:torsion}, $ E $ is given by the Weierstrass equation
$$ y^2 + \dfrac{f_4}{(f - 1)^3}xy - \dfrac{f(2f - 1)f_1f_2}{(f - 1)^4}y = x^3 - \dfrac{f(2f - 1)f_1f_2}{(f - 1)^4}x^2, $$
for some $ f \in k(t) \setminus k $, with
$$ \Delta = \dfrac{f^{12}(2f - 1)^6f_1^3f_2^4f_3}{(f - 1)^{24}}, \qquad
\begin{aligned}
f_1 & := 2f^2 - 2f + 1, \\
f_2 & := 3f^2 - 3f + 1, \\
f_3 & := 6f^2 - 6f + 1, \\
f_4 & := 6f^4 - 8f^3 + 2f^2 + 2f - 1.
\end{aligned}
$$
There are distinct places $ v_0 $ and $ v_1 $ such that $ v_0(f) > 0 $ and $ v_1(f - 1) > 0 $. At $ v_0 $,
$$ v_0(a_1) = 0, \qquad v_0(a_2) = v_0(a_3) = v_0(f) > 0, \qquad v_0(\Delta) = 12v_0(f). $$
By Lemma \ref{lem:odd}, $ c_{v_0}(E) = 12v_0(f) $, which gives a factor $ 12 \mid c(E) $. Now the change of variables
$$ (x, y) \mapsto \left(\dfrac{(f - 1)^6}{f^8}x, \dfrac{(f - 1)^9}{f^{12}}y\right) $$
transforms $ E $ to the Weierstrass equation
$$ y^2 + \dfrac{f_4}{f^4}xy - \dfrac{(f - 1)^5(2f - 1)f_1f_2}{f^{11}}y = x^3 - \dfrac{(f - 1)^2(2f - 1)f_1f_2}{f^7}x^2, $$
with
$$ \Delta = \dfrac{(f - 1)^{12}(2f - 1)^6f_1^3f_2^4f_3}{f^{36}}. $$
At $ v_1 $,
$$ v_1(a_1) = 0, \qquad
\begin{aligned}
v_1(a_2) & = 2v_1(f - 1) > 0, \\
v_1(a_3) & = 5v_1(f - 1) > 0,
\end{aligned}
\qquad v_1(\Delta) = 12v_1(f - 1). $$
By Lemma \ref{lem:odd}, $ c_{v_1}(E) = 12v_1(f - 1) $, which gives another factor $ 12 \mid c(E) $.

\pagebreak

Let $ p \ne 2 $. Since $ f_1 \notin k $ by Lemma \ref{lem:negative}, there are places $ v_2 $ and $ v_3 $, distinct from $ v_0 $ and $ v_1 $, such that $ v_2(2f - 1) > 0 $ and $ v_3(f_1) > 0 $. At $ v_2 $,
$$ v_2(a_1) = 0, \qquad v_2(a_2) = v_2(a_3) = v_2(2f - 1) > 0, \qquad v_2(\Delta) = 6v_2(2f - 1). $$
At $ v_3 $, Lemma \ref{lem:identity} \footnote{$ f_2f_3 = 9(f^2 - f)f_1 + 1 $ and $ 2f_4 = (6f^2 - 2f - 3)f_1 + 1 $} gives
$$ v_3(f) = v_3(f - 1) = v_3(2f - 1) = v_3(f_2) = v_3(f_3) = v_3(f_4) = 0, $$
so
$$ v_3(a_1) = 0, \qquad v_3(a_2) = v_3(a_3) = v_3(f_1) > 0, \qquad v_3(\Delta) = 3v_3(f_1). $$
By Lemma \ref{lem:odd}, $ c_{v_2}(E) = 6v_2(2f - 1) $ and $ c_{v_3}(E) = 3v_3(f_1) $, which give a factor $ 6 \mid c(E) $ and a factor $ 3 \mid c(E) $.

Now let $ p \ne 3 $. Since $ f_2 \notin k $ by Lemma \ref{lem:negative}, there is a place $ v_4 $, distinct from $ v_0, \dots, v_3 $, such that $ v_4(f_2) > 0 $. At $ v_4 $, Lemma \ref{lem:identity} \footnote{$ 3f_1f_3 = 4(3f^2 - 3f + 1)f_2 - 1 $ and $ 3f_4 = 2(3f^2 - f - 1)f_2 + (2f - 1) $} gives
$$ v_4(f) = v_4(f - 1) = v_4(2f - 1) = v_4(f_1) = v_4(f_3) = v_4(f_4) = 0, $$
so
$$ v_4(a_1) = 0, \qquad v_4(a_2) = v_4(a_3) = v_4(f_2) > 0, \qquad v_4(\Delta) = 4v_4(f_2). $$
By Lemma \ref{lem:odd}, $ c_{v_4}(E) = 4v_4(f_2) $, which gives a factor $ 4 \mid c(E) $.

Now let $ p \ne 2, 3 $. There is a place $ v_5 $, distinct from $ v_0, \dots, v_4 $, such that $ v_5(f) < 0 $. At $ v_5 $, Lemma \ref{lem:negative} gives
$$ v_5(c_4) = 2v_5(b_2) = 2v_5(12) = 0, \qquad v_5(\Delta) = -2v_5(f). $$
By Lemma \ref{lem:even}, $ 2 \mid c_{v_5}(E) $, which gives a factor $ 2 \mid c(E) $.
\end{proof}

\begin{example}
The elliptic curve $ E / \F_{11}(t) $ given by the parameterisation in Proposition \ref{prop:c12} with $ f = t $ has $ c(E) = 12^4 $. If $ \F_8 := \F_2[\alpha] / (\alpha^3 + \alpha + 1) $, then the elliptic curve $ E / \F_8(t) $ given by the same parameterisation with $ f = (\alpha^2 + \alpha)t^3 + t^2 + \alpha $ has $ c(E) = 2^2 \cdot 12^2 $. The elliptic curve $ E / \F_3(t) $ given by the same parameterisation with $ f = t $ has $ c(E) = 3 \cdot 6 \cdot 12^2 $.
\end{example}

\begin{proposition}
\label{prop:c14}
Let $ k $ be a finite field of characteristic $ p $, and let $ E / k(t) $ be a non-isotrivial elliptic curve with $ E(k(t))_{\tors} \cong \C_{14} $.
\begin{enumerate}
\item If $ p = 2 $, then $ 2 \cdot 14^3 $ divides $ c(E) $.
\item If $ p = 7 $, then $ 7^3 \cdot 14^3 $ divides $ c(E) $.
\end{enumerate}
\end{proposition}

\begin{proof}
Let $ p = 2 $. By Theorem \ref{thm:torsion}, $ E $ is given by the Weierstrass equation
$$ y^2 + \dfrac{f_2^2}{f_1}xy - \dfrac{f(f + 1)^3}{f_1^2}y = x^3 - \dfrac{f(f + 1)^3}{f_1^2}x^2, $$
for some $ f \in k(t) \setminus k $, with
$$ \Delta = \dfrac{f^{14}(f + 1)^{14}}{f_1^{10}}, \qquad f_1 := f^3 + f + 1, \qquad f_2 := f^2 + f + 1. $$

\pagebreak

\noindent Since $ f_1 \notin k $ by Lemma \ref{lem:negative}, there are distinct places $ v_0, \dots, v_3 $ such that
$$ v_0(f) > 0, \qquad v_1(f + 1) > 0, \qquad v_2(f_1) > 0, \qquad v_3(f) < 0. $$
At $ v_0 $ and $ v_1 $,
$$
\begin{aligned}
v_0(a_1) & = 0, & v_0(a_2) & = v_0(a_3) = v_0(f) > 0, & v_0(\Delta) & = 14v_0(f), \\
v_1(a_1) & = 0, & v_1(a_2) & = v_1(a_3) = 3v_1(f + 1) > 0, & v_1(\Delta) & = 14v_1(f + 1).
\end{aligned}
$$
By Lemma \ref{lem:odd}, $ c_{v_0}(E) = 14v_0(f) $ and $ c_{v_1}(E) = 14v_1(f + 1) $, which give two factors $ 14 \mid c(E) $. Now the change of variables
$$ (x, y) \mapsto \left(\dfrac{f_1^2}{f^8}x, \dfrac{f_1^3}{f^{12}}y\right) $$
transforms $ E $ to the Weierstrass equation
$$ y^2 + \dfrac{f_2^2}{f^4}xy - \dfrac{(f + 1)^3f_1}{f^{11}}y = x^3 - \dfrac{(f + 1)^3}{f^7}x^2, $$
with $ \Delta = (f + 1)^{14}f_1^2 / f^{34} $. At $ v_3 $, Lemma \ref{lem:negative} gives
$$ v_3(a_1) = 0, \qquad
\begin{aligned}
v_3(a_2) & = -4v_3(f) > 0, \\
v_3(a_3) & = -5v_3(f) > 0,
\end{aligned}
\qquad v_3(\Delta) = -14v_3(f). $$
By Lemma \ref{lem:odd}, $ c_{v_3}(E) = -14v_3(f) $, which gives another factor $ 14 \mid c(E) $. At $ v_2 $, Lemma \ref{lem:identity} \footnote{$ f_2 = f_1 + f^2(f + 1) $} gives
$$ v_2(f) = v_2(f + 1) = v_2(f_2) = 0, $$
so $ v_2(c_4) = 4v_2(a_1) = 0 $ and $ v_2(\Delta) = 2v_2(f_1) $. By Lemma \ref{lem:even}, $ 2 \mid c_{v_2}(E) $, which gives a factor $ 2 \mid c(E) $.

Now let $ p = 7 $. By Theorem \ref{thm:torsion}, $ E $ is given by the Weierstrass equation
$$
\begin{aligned}
& y^2 + \dfrac{6f_3f_4}{f(f + 2)^2(f + 5)}xy - \dfrac{(f + 1)^2(f + 3)^3(f + 4)(f + 5)^2(f + 6)}{4f^2(f + 2)^3}y \\
& \qquad = x^3 - \dfrac{(f + 1)^2(f + 3)^3(f + 4)(f + 5)^2(f + 6)}{4f^2(f + 2)^3}x^2,
\end{aligned}
$$
for some $ f \in k(t) \setminus k $, with
$$ \Delta = \dfrac{(f + 1)^7(f + 3)^{14}(f + 4)^7(f + 5)^2(f + 6)^{14}}{f^9(f + 2)^{17}}, $$
where $ f_3 := f^3 + 2f^2 + f + 4 $ and $ f_4 := f^3 + 4f^2 + 3f + 1 $. There are distinct places $ v_1, \dots, v_6 $ such that
$$
\begin{aligned}
v_1(f + 1) & > 0, & v_2(f + 2) & > 0, & v_3(f + 3) & > 0, \\
v_4(f + 4) & > 0, & v_5(f + 5) & > 0, & v_6(f + 6) & > 0.
\end{aligned}
$$

\pagebreak

\noindent At $ v_1, v_3, v_4, v_6 $,
$$
\begin{aligned}
v_1(a_1) & = 0, & v_1(a_2) & = v_1(a_3) = 2v_1(f + 1) > 0, & v_1(\Delta) & = 7v_1(f + 1), \\
v_3(a_1) & = 0, & v_3(a_2) & = v_3(a_3) = 3v_3(f + 3) > 0, & v_3(\Delta) & = 14v_3(f + 3), \\
v_4(a_1) & = 0, & v_4(a_2) & = v_4(a_3) = v_4(f + 4) > 0, & v_4(\Delta) & = 7v_4(f + 4), \\
v_6(a_1) & = 0, & v_6(a_2) & = v_6(a_3) = v_6(f + 6) > 0, & v_6(\Delta) & = 14v_6(f + 6).
\end{aligned}
$$
By Lemma \ref{lem:odd},
$$
\begin{aligned}
c_{v_1}(E) & = 7v_1(f + 1), & c_{v_3}(E) & = 14v_3(f + 3), \\
c_{v_4}(E) & = 7v_4(f + 4), & c_{v_6}(E) & = 14v_6(f + 6),
\end{aligned}
$$
which give two factors $ 7 \mid c(E) $ and two factors $ 14 \mid c(E) $. Now the change of variables
$$ (x, y) \mapsto (f^2(f + 2)^4(f + 5)^2x, f^3(f + 2)^6(f + 5)^3y) $$
transforms $ E $ to the Weierstrass equation
$$
\begin{aligned}
& y^2 + 6f_3f_4xy - \tfrac{1}{4}f(f + 1)^2(f + 2)^3(f + 3)^3(f + 4)(f + 5)^5(f + 6)y \\
& \qquad = x^3 - \tfrac{1}{4}(f + 1)^2(f + 2)(f + 3)^3(f + 4)(f + 5)^4(f + 6)x^2,
\end{aligned}
$$
with
$$ \Delta = f^3(f + 1)^7(f + 2)^7(f + 3)^{14}(f + 4)^7(f + 5)^{14}(f + 6)^{14}. $$
At $ v_2 $ and $ v_5 $,
$$
\begin{aligned}
v_2(a_1) & = 0, &
\begin{aligned}
v_2(a_2) & = v_2(f + 2) > 0, \\
v_2(a_3) & = 3v_2(f + 2) > 0,
\end{aligned}
& & v_2(\Delta) & = 7v_2(f + 2), \\
v_5(a_1) & = 0, &
\begin{aligned}
v_5(a_2) & = 4v_5(f + 5) > 0, \\
v_5(a_3) & = 5v_5(f + 5) > 0,
\end{aligned}
& & v_5(\Delta) & = 14v_5(f + 5).
\end{aligned}
$$
By Lemma \ref{lem:odd}, $ c_{v_2}(E) = 7v_2(f + 2) $ and $ c_{v_5}(E) = 14v_5(f + 5) $, which give another factor $ 7 \mid c(E) $ and another factor $ 14 \mid c(E) $.
\end{proof}

\begin{example}
The elliptic curve $ E / \F_2(t) $ given by its parameterisation in Proposition \ref{prop:c14} with $ f = t $ has $ c(E) = 2 \cdot 14^3 $. If $ \F_{49} := \F_7[\alpha] / (\alpha^2 + 6\alpha + 3) $, then the elliptic curve $ E / \F_{49}(t) $ given by its parameterisation in Proposition \ref{prop:c14} with $ f = \alpha t^4 $ has $ c(E) = 7^3 \cdot 14^3 $.
\end{example}

\begin{proposition}
\label{prop:c18}
Let $ k $ be a finite field of characteristic $ p = 2 $, and let $ E / k(t) $ be a non-isotrivial elliptic curve with $ E(k(t))_{\tors} \cong \C_{18} $. Then $ 2^2 \cdot 3 \cdot 18^3 $ divides $ c(E) $.
\end{proposition}

\begin{proof}
By Theorem \ref{thm:torsion}, $ E $ is given by the Weierstrass equation
$$ y^2 + \dfrac{f_3^2}{f_2}xy - \dfrac{f(f + 1)^5f_1}{f_2^2}y = x^3 - \dfrac{f(f + 1)^5f_1}{f_2^2}x^2, $$

\pagebreak

\noindent for some $ f \in k(t) \setminus k $, with
$$ \Delta = \dfrac{f^{18}(f + 1)^{18}f_1^6}{f_2^{10}}, \qquad
\begin{aligned}
f_1 & := f^2 + f + 1, \\
f_2 & := f^3 + f + 1, \\
f_3 & := f^3 + f^2 + 1.
\end{aligned}
$$
Since $ f_1, f_2 \notin k $ by Lemma \ref{lem:negative}, there are distinct places $ v_0, \dots, v_4 $ such that
$$ v_0(f) > 0, \quad v_1(f + 1) > 0, \quad v_2(f_1) > 0, \quad v_3(f_2) > 0, \quad v_4(f) < 0. $$
At $ v_0 $ and $ v_1 $,
$$
\begin{aligned}
v_0(a_1) & = 0, & v_0(a_2) & = v_0(a_3) = v_0(f) > 0, & v_0(\Delta) & = 18v_0(f), \\
v_1(a_1) & = 0, & v_1(a_2) & = v_1(a_3) = 5v_1(f + 1) > 0, & v_1(\Delta) & = 18v_1(f + 1).
\end{aligned}
$$
At $ v_2 $, Lemma \ref{lem:identity} \footnote{$ f_2f_3 = (f^4 + f)f_1 + 1 $} gives
$$ v_2(f) = v_2(f + 1) = v_2(f_2) = v_2(f_3) = 0, $$
so
$$ v_2(a_1) = 0, \qquad v_2(a_2) = v_2(a_3) = v_2(f_1) > 0, \qquad v_2(\Delta) = 6v_2(f_1). $$
By Lemma \ref{lem:odd},
$$ c_{v_0}(E) = 18v_0(f), \qquad c_{v_1}(E) = 18v_1(f + 1), \qquad c_{v_2}(E) = 6v_2(f_1), $$
which give two factors $ 18 \mid c(E) $ and a factor $ 6 \mid c(E) $. Now the change of variables
$$ (x, y) \mapsto \left(\dfrac{f_2^2}{f^{12}}x, \dfrac{f_2^3}{f^{18}}y\right) $$
transforms $ E $ to the Weierstrass equation
$$ y^2 + \dfrac{f_3^2}{f^6}xy - \dfrac{(f + 1)^5f_1f_2}{f^{17}}y = x^3 - \dfrac{(f + 1)^5f_1}{f^{11}}x^2, $$
with $ \Delta = (f + 1)^{18}f_1^6f_2^2 / f^{54} $. At $ v_4 $, Lemma \ref{lem:negative} gives
$$ v_4(a_1) = 0, \qquad
\begin{aligned}
v_4(a_2) & = -4v_4(f) > 0, \\
v_4(a_3) & = -7v_4(f) > 0,
\end{aligned}
\qquad v_4(\Delta) = -18v_4(f). $$
By Lemma \ref{lem:odd}, $ c_{v_4}(E) = -18v_4(f) $, which gives another factor $ 18 \mid c(E) $. At $ v_3 $, Lemma \ref{lem:identity} \footnote{$ f_1f_3 = (f^2 + 1)f_2 + f^2 $} gives
$$ v_3(f) = v_3(f + 1) = v_3(f_1) = v_3(f_3) = 0, $$
so $ v_3(c_4) = 4v_3(a_1) = 0 $ and $ v_3(\Delta) = 2v_3(f_2) $. By Lemma \ref{lem:even}, $ 2 \mid c_{v_3}(E) $, which gives a factor $ 2 \mid c(E) $.
\end{proof}

\pagebreak

\begin{example}
The elliptic curve $ E / \F_2(t) $ given by the parameterisation in Proposition \ref{prop:c18} with $ f = t $ has $ c(E) = 2^2 \cdot 3 \cdot 18^3 $.
\end{example}

\section{Non-cyclic torsion}
\label{sec:noncyclic}

For a complete discretely valued field $ F $ with residue field $ \kappa $ and an elliptic curve $ E / F $ given by a minimal Weierstrass equation, recall that there is a surjective reduction map $ E_0(F) \to \widetilde{E}_{\ns}(\kappa) $ with kernel $ E_1(F) $, where $ \widetilde{E}_{\ns}(\kappa) $ is the group of non-singular points on the reduction $ \widetilde{E} / \kappa $ and $ E_0(F) $ is its corresponding subgroup of $ E(F) $ \cite[Proposition VII.2.1]{Sil09}.

\begin{lemma}
\label{lem:noncyclic}
Let $ k $ be a finite field of characteristic $ p $, and let $ E / k(t) $ be an elliptic curve such that $ E[n] \subseteq E(k(t)) $ for some integer $ n > 2 $ with $ p \nmid n $. Then $ \zeta_n \in k $. Furthermore, if $ E $ is given by an integral Weierstrass equation at a place $ v $ with $ v(c_4) = 0 $ and $ v(\Delta) > 0 $, then the equation is minimal with split multiplicative reduction at $ v $ of type $ \I_{v(\Delta)} $, and $ c_v(E) = v(\Delta) $.
\end{lemma}

\begin{proof}
The Weil pairing gives $ \zeta_n \in k(t) $ \cite[Corollary III.8.1.1]{Sil09}, so $ \zeta_n \in k $ \cite[Lemma 33.8.6]{Stacks}. Since $ v(c_4) = 0 $, the equation is minimal and Tate's algorithm terminates at step $ 2 $, so the reduction is multiplicative of type $ \I_{v(\Delta)} $ \cite[Page 366]{Sil94}. Suppose that it were non-split, so that $ c_v(E) \le 2 $. Then $ E_0(k(t)_v) / E_1(k(t)_v) $ is the kernel of the norm map $ \F_{q_v^2}^\times \to \F_{q_v}^\times $, where $ q_v := |k_v| $, which is isomorphic to $ \C_{q_v + 1} $. Now $ E_1(k(t)_v)[n] = 0 $ since $ p \nmid n $ \cite[Proposition VII.3.1]{Sil09}, so $ E_0(k(t)_v)[n] $ injects into $ \C_{q_v + 1} $, and hence $ |E_0(k(t)_v)[n]| \le n $. By the assumption that $ E[n] \subseteq E(k(t)) $,
$$ n^2 = |E(k(t)_v)[n]| \le c_v(E)|E_0(k(t)_v)[n]| \le 2n, $$
which contradicts $ n > 2 $. Thus the reduction is split, and $ c_v(E) = v(\Delta) $.
\end{proof}

\begin{proposition}
\label{prop:c3c3}
Let $ k $ be a finite field of characteristic $ p \ne 3 $, and let $ E / k(t) $ be a non-isotrivial elliptic curve with $ E(k(t))_{\tors} \cong \C_3 \oplus \C_3 $. Then $ 3^4 $ divides $ c(E) $.
\end{proposition}

\begin{proof}
By Lemma \ref{lem:noncyclic}, $ \zeta_3 \in k $. By Theorem \ref{thm:torsion}, $ E $ is given by the Weierstrass equation
$$ y^2 + 3(f + 2)xy + 9(f - \zeta_3)(f - \zeta_3^2)y = x^3, $$
for some $ f \in k(t) \setminus k $, with $ \Delta = 3^9(f - 1)^3(f - \zeta_3)^3(f - \zeta_3^2)^3 $. There are distinct places $ v_1, \dots, v_4 $ such that
$$ v_1(f - 1) > 0, \qquad v_2(f - \zeta_3) > 0, \qquad v_3(f - \zeta_3^2) > 0, \qquad v_4(f) < 0. $$
At $ v_2 $ and $ v_3 $,
$$
\begin{aligned}
v_2(a_1) & = 0, & v_2(a_3) & = v_2(f - \zeta_3) > 0, & v_2(\Delta) & = 3v_2(f - \zeta_3), \\
v_3(a_1) & = 0, & v_3(a_3) & = v_3(f - \zeta_3^2) > 0, & v_3(\Delta) & = 3v_3(f - \zeta_3^2).
\end{aligned}
$$

\pagebreak

\noindent By Lemma \ref{lem:odd}, $ c_{v_2}(E) = 3v_2(f - \zeta_3) $ and $ c_{v_3}(E) = 3v_3(f - \zeta_3^2) $, which give two factors $ 3 \mid c(E) $. At $ v_1 $,
$$ v_1(c_4) = v_1(3^8 - 8 \cdot 3^6) = v_1(3^6) = 0, \qquad v_1(\Delta) = 3v_1(f - 1). $$
By Lemma \ref{lem:noncyclic}, $ c_{v_1}(E) = 3v_1(f - 1) $, which gives another factor $ 3 \mid c(E) $. Now the change of variables $ (x, y) \mapsto (x / f^2, y / f^3) $ transforms $ E $ to the Weierstrass equation
$$ y^2 + \dfrac{3(f + 2)}{f}xy + \dfrac{9(f - \zeta_3)(f - \zeta_3^2)}{f^3}y = x^3, $$
with $ \Delta = 3^9(f - 1)^3(f - \zeta_3)^3(f - \zeta_3^2)^3 / f^{12} $. At $ v_4 $, Lemma \ref{lem:negative} gives
$$ v_4(a_1) = 0, \qquad v_4(a_3) = -v_4(f) > 0, \qquad v_4(\Delta) = -3v_4(f). $$
By Lemma \ref{lem:odd}, $ c_{v_4}(E) = -3v_4(f) $, which gives another factor $ 3 \mid c(E) $.
\end{proof}

\begin{example}
If $ \F_4 := \F_2[\alpha] / (\alpha^2 + \alpha + 1) $, then the elliptic curve $ E / \F_4(t) $ given by the parameterisation in Proposition \ref{prop:c3c3} with $ f = t $ has $ c(E) = 3^4 $.
\end{example}

\begin{proposition}
\label{prop:c4c2}
Let $ k $ be a finite field of characteristic $ p \ne 2 $, and let $ E / k(t) $ be a non-isotrivial elliptic curve with $ E(k(t))_{\tors} \cong \C_4 \oplus \C_2 $. Then $ 4^2 \cdot 2^2 $ divides $ c(E) $.
\end{proposition}

\begin{proof}
By Theorem \ref{thm:torsion}, $ E $ is given by the Weierstrass equation
$$ y^2 + xy - (f - \tfrac{1}{4})(f + \tfrac{1}{4})y = x^3 - (f - \tfrac{1}{4})(f + \tfrac{1}{4})x^2, $$
for some $ f \in k(t) \setminus k $, with $ \Delta = 16f^2(f - \tfrac{1}{4})^4(f + \tfrac{1}{4})^4 $. There are distinct places $ v_1, \dots, v_4 $ such that
$$ v_1(f) > 0, \qquad v_2(f - \tfrac{1}{4}) > 0, \qquad v_3(f + \tfrac{1}{4}) > 0, \qquad v_4(f) < 0. $$
At $ v_2 $ and $ v_3 $,
$$
\begin{aligned}
v_2(a_1) & = 0, & v_2(a_2) & = v_2(a_3) = v_2(f - \tfrac{1}{4}) > 0, & v_2(\Delta) & = 4v_2(f - \tfrac{1}{4}), \\
v_3(a_1) & = 0, & v_3(a_2) & = v_3(a_3) = v_3(f + \tfrac{1}{4}) > 0, & v_3(\Delta) & = 4v_3(f + \tfrac{1}{4}).
\end{aligned}
$$
By Lemma \ref{lem:odd}, $ c_{v_2}(E) = 4v_2(f - \tfrac{1}{4}) $ and $ c_{v_3}(E) = 4v_3(f + \tfrac{1}{4}) $, which give two factors $ 4 \mid c(E) $. At $ v_1 $,
$$ v_1(c_4) = v_1((\tfrac{5}{4})^2 - \tfrac{3}{2}) = v_1(\tfrac{1}{16}) = 0, \qquad v_1(\Delta) = 2v_1(f). $$
By Lemma \ref{lem:even}, $ 2 \mid c_{v_1}(E) $, which gives a factor $ 2 \mid c(E) $. Now the change of variables $ (x, y) \mapsto (x / f^2, y / f^3) $ transforms $ E $ to the Weierstrass equation
$$ y^2 + \dfrac{1}{f}xy - \dfrac{(f - \tfrac{1}{4})(f + \tfrac{1}{4})}{f^3}y = x^3 - \dfrac{(f - \tfrac{1}{4})(f + \tfrac{1}{4})}{f^2}x^2, $$
with $ \Delta = 16(f - \tfrac{1}{4})^4(f + \tfrac{1}{4})^4 / f^{10} $. At $ v_4 $, Lemma \ref{lem:negative} gives
$$ v_4(c_4) = 2v_4(b_2) = 2v_4(4) = 0, \qquad v_4(\Delta) = -2v_4(f). $$
By Lemma \ref{lem:even}, $ 2 \mid c_{v_4}(E) $, which gives another factor $ 2 \mid c(E) $.
\end{proof}

\pagebreak

\begin{example}
The elliptic curve $ E / \F_3(t) $ given by the parameterisation in Proposition \ref{prop:c4c2} with $ f = t $ has $ c(E) = 4^2 \cdot 2^2 $.
\end{example}

\begin{proposition}
\label{prop:c6c2}
Let $ k $ be a finite field of characteristic $ p \ne 2 $, and let $ E / k(t) $ be a non-isotrivial elliptic curve with $ E(k(t))_{\tors} \cong \C_6 \oplus \C_2 $.
\begin{enumerate}
\item If $ p \ne 3 $, then $ 2^3 \cdot 6^3 $ divides $ c(E) $.
\item If $ p = 3 $, then $ 6^3 $ divides $ c(E) $.
\end{enumerate}
\end{proposition}

\begin{proof}
By Theorem \ref{thm:torsion}, $ E $ is given by the Weierstrass equation
$$ y^2 + \dfrac{f_1}{(f - 3)(f + 3)}xy + \dfrac{2(f - 5)(f - 1)^2}{(f - 3)^2(f + 3)^2}y = x^3 + \dfrac{2(f - 5)(f - 1)^2}{(f - 3)^2(f + 3)^2}x^2, $$
for some $ f \in k(t) \setminus k $, with
$$ \Delta = \dfrac{64(f - 5)^6(f - 1)^6(f - 9)^2}{(f - 3)^{10}(f + 3)^{10}}, \qquad f_1 := f^2 + 2f - 19. $$
There are distinct places $ v_1, v_2, v_3 $ such that
$$ v_1(f - 5) > 0, \qquad v_2(f - 1) > 0, \qquad v_3(f) < 0. $$
At $ v_1 $ and $ v_2 $,
$$
\begin{aligned}
v_1(a_1) & = 0, & v_1(a_2) & = v_1(a_3) = v_1(f - 5) > 0, & v_1(\Delta) & = 6v_1(f - 5), \\
v_2(a_1) & = 0, & v_2(a_2) & = v_2(a_3) = 2v_2(f - 1) > 0, & v_2(\Delta) & = 6v_2(f - 1).
\end{aligned}
$$
At $ v_3 $, Lemma \ref{lem:negative} gives
$$ v_3(a_1) = 0, \qquad v_3(a_2) = v_3(a_3) = -v_3(f) > 0, \qquad v_3(\Delta) = -6v_3(f). $$
By Lemma \ref{lem:odd},
$$ c_{v_1}(E) = 6v_1(f - 5), \qquad c_{v_2}(E) = 6v_2(f - 1), \qquad c_{v_3}(E) = -6v_3(f), $$
which give three factors $ 6 \mid c(E) $.

Let $ p \ne 3 $. There are places $ v_4, v_5, v_6 $, distinct from $ v_1, v_2, v_3 $, such that
$$ v_4(f - 9) > 0, \qquad v_5(f - 3) > 0, \qquad v_6(f + 3) > 0. $$
At $ v_4 $, we have $ v_4(c_4) = v_4(\tfrac{2^4}{3^6}) = 0 $ and $ v_4(\Delta) = 2v_4(f - 9) $. By Lemma \ref{lem:even}, $ 2 \mid c_{v_4}(E) $, which gives a factor $ 2 \mid c(E) $. Now the change of variables
$$ (x, y) \mapsto ((f - 3)^2(f + 3)^2x, (f - 3)^3(f + 3)^3y) $$
transforms $ E $ to the Weierstrass equation
$$ y^2 + f_1xy + 2(f - 5)(f - 1)^2(f - 3)(f + 3)y = x^3 + 2(f - 5)(f - 1)^2x^2, $$
with
$$ \Delta = 64(f - 5)^6(f - 1)^6(f - 9)^2(f - 3)^2(f + 3)^2. $$
At $ v_5 $ and $ v_6 $,
$$
\begin{aligned}
v_5(c_4) & = 2v_5(b_2) = 2v_5(48) = 0, & v_5(\Delta) & = 2v_5(f - 3), \\
v_6(c_4) & = 2v_6(b_2) = 2v_6(768) = 0, & v_6(\Delta) & = 2v_6(f + 3).
\end{aligned}
$$

\pagebreak

\noindent By Lemma \ref{lem:even}, $ 2 \mid c_{v_5}(E) $ and $ 2 \mid c_{v_6}(E) $, which give two more factors $ 2 \mid c(E) $.
\end{proof}

\begin{example}
The elliptic curve $ E / \F_5(t) $ given by the parameterisation in Proposition \ref{prop:c6c2} with $ f = t $ has $ c(E) = 2^3 \cdot 6^3 $. If $ \F_{27} := \F_3[\alpha] / (\alpha^3 + 2\alpha + 1) $, then the elliptic curve $ E / \F_{27}(t) $ given by the same parameterisation with $ f = (\alpha^2 + 2)t^2 / (t^2 + 1) $ has $ c(E) = 6^3 $.
\end{example}

\begin{proposition}
\label{prop:c8c2}
Let $ k $ be a finite field of characteristic $ p \ne 2 $, and let $ E / k(t) $ be a non-isotrivial elliptic curve with $ E(k(t))_{\tors} \cong \C_8 \oplus \C_2 $. Then $ 2^{16} $ divides $ c(E) $.
\end{proposition}

\begin{proof}
By Theorem \ref{thm:torsion}, $ E $ is given by the Weierstrass equation
$$ y^2 + \dfrac{f_3}{2f(4f + 1)(8f^2 - 1)}xy - \dfrac{(2f + 1)f_1}{(8f^2 - 1)^2}y = x^3 - \dfrac{(2f + 1)f_1}{(8f^2 - 1)^2}x^2, $$
for some $ f \in k(t) \setminus k $, with
$$ \Delta = \dfrac{(2f + 1)^8f_1^4f_2^2}{2^4f^4(4f + 1)^4(8f^2 - 1)^{10}}, \qquad
\begin{aligned}
f_1 & := 8f^2 + 4f + 1, \\
f_2 & := 8f^2 + 8f + 1, \\
f_3 & := 64f^4 - 24f^2 - 8f - 1.
\end{aligned}
$$
Since $ 8f^2 - 1, f_1, f_2 \notin k $ by Lemma \ref{lem:negative}, there are distinct places $ v_1, \dots, v_7 $ such that
$$
\begin{aligned}
v_1(f) & > 0, & v_2(2f + 1) & > 0, & v_3(4f + 1) & > 0, & v_4(8f^2 - 1) & > 0, \\
v_5(f_1) & > 0, & v_6(f_2) & > 0, & v_7(f) & < 0.
\end{aligned}
$$
At $ v_2 $,
$$ v_2(a_1) = 0, \qquad v_2(a_2) = v_2(a_3) = v_2(2f + 1) > 0, \qquad v_2(\Delta) = 8v_2(2f + 1). $$
At $ v_5 $, Lemma \ref{lem:identity} \footnote{$ (8f^2 - 1)f_2 = (8f^2 + 4f - 3)f_1 + 2 $ and $ f_3 = 2(4f^2 - 2f - 1)f_1 + (4f + 1) $} gives
$$ v_5(f) = v_5(2f + 1) = v_5(4f + 1) = v_5(8f^2 - 1) = v_5(f_2) = v_5(f_3) = 0, $$
so
$$ v_5(a_1) = 0, \qquad v_5(a_2) = v_5(a_3) = v_5(f_1) > 0, \qquad v_5(\Delta) = 4v_5(f_1). $$
At $ v_7 $, Lemma \ref{lem:negative} gives
$$ v_7(a_1) = 0, \qquad v_7(a_2) = v_7(a_3) = -v_7(f) > 0, \qquad v_7(\Delta) = -8v_7(f). $$
By Lemma \ref{lem:odd},
$$ c_{v_2}(E) = 8v_2(2f + 1), \qquad c_{v_5}(E) = 4v_5(f_1), \qquad c_{v_7}(E) = -8v_7(f), $$

\pagebreak

\noindent which give two factors $ 8 \mid c(E) $ and a factor $ 4 \mid c(E) $. At $ v_6 $, Lemma \ref{lem:identity} \footnote{$ (8f^2 - 1)f_1 = (8f^2 - 4f - 1)f_2 + 8f(4f + 1) $ and $ 2^4f^4(4f + 1)^4(8f^2 - 1)^4c_4 = 2^8(8192f^{14} + 24576f^{13} + 31744f^{12} + 22528f^{11} + 8832f^{10} + 1152f^9 - 656f^8 - 416f^7 - 112f^6 - 16f^5 - f^4)f_2 + (4f + 1)^8 $} gives
$$
\begin{aligned}
& v_6(f) = v_6(2f + 1) = v_6(4f + 1) = v_6(8f^2 - 1) = 0, \\
& v_6(f_1) = v_6(2^4f^4(4f + 1)^4(8f^2 - 1)^4c_4) = 0,
\end{aligned}
$$
so $ v_6(c_4) = 0 $ and $ v_6(\Delta) = 2v_6(f_2) $. By Lemma \ref{lem:even}, $ 2 \mid c_{v_6}(E) $, which gives a factor $ 2 \mid c(E) $. Now the change of variables
$$ (x, y) \mapsto (2^2f^2(4f + 1)^2(8f^2 - 1)^2x, 2^3f^3(4f + 1)^3(8f^2 - 1)^3y) $$
transforms $ E $ to the Weierstrass equation
$$
\begin{aligned}
& y^2 + f_3xy - 2^3f^3(2f + 1)(4f + 1)^3(8f^2 - 1)f_1y \\
& \qquad = x^3 - 2^2f^2(2f + 1)(4f + 1)^2f_1x^2,
\end{aligned}
$$
with
$$ \Delta = 2^8f^8(2f + 1)^8(4f + 1)^8(8f^2 - 1)^2f_1^4f_2^2. $$
At $ v_1 $ and $ v_3 $,
$$
\begin{aligned}
v_1(a_1) & = 0, &
\begin{aligned}
v_1(a_2) & = 2v_1(f) > 0, \\
v_1(a_3) & = 3v_1(f) > 0,
\end{aligned}
& & v_1(\Delta) & = 8v_1(f), \\
v_3(a_1) & = 0, &
\begin{aligned}
v_3(a_2) & = 2v_3(4f + 1) > 0, \\
v_3(a_3) & = 3v_3(4f + 1) > 0,
\end{aligned}
& & v_3(\Delta) & = 8v_3(4f + 1).
\end{aligned}
$$
By Lemma \ref{lem:odd}, $ c_{v_1}(E) = 8v_1(f) $ and $ c_{v_3}(E) = 8v_3(4f + 1) $, which give two more factors $ 8 \mid c(E) $. At $ v_4 $, Lemma \ref{lem:identity} \footnote{$ f_1f_2 = (8f^2 + 12f + 3)(8f^2 - 1) + 4(2f + 1)(4f + 1) $ and $ b_2 = 2(256f^6 - 256f^5 - 544f^4 - 336f^3 - 104f^2 - 16f - 1)(8f^2 - 1) - (4f + 1)^4 $} gives
$$ v_4(f) = v_4(2f + 1) = v_4(4f + 1) = v_4(f_1) = v_4(f_2) = v_4(b_2) = 0, $$
so $ v_4(c_4) = 2v_4(b_2) = 0 $ and $ v_4(\Delta) = 2v_4(8f^2 - 1) $. By Lemma \ref{lem:even}, $ 2 \mid c_{v_4}(E) $, which gives another factor $ 2 \mid c(E) $.
\end{proof}

\begin{example}
The elliptic curve $ E / \F_3(t) $ given by the parameterisation in Proposition \ref{prop:c8c2} with $ f = t $ has $ c(E) = 2^{16} $.
\end{example}

\begin{proposition}
\label{prop:c4c4}
Let $ k $ be a finite field of characteristic $ p \ne 2 $, and let $ E / k(t) $ be a non-isotrivial elliptic curve with $ E(k(t))_{\tors} \cong \C_4 \oplus \C_4 $. Then $ 4^6 $ divides $ c(E) $.
\end{proposition}

\pagebreak

\begin{proof}
By Lemma \ref{lem:noncyclic}, $ \zeta_4 \in k $. By Theorem \ref{thm:torsion}, $ E $ is given by the Weierstrass equation
$$ y^2 + xy - \prod_{i = 1}^4(f - \tfrac{\zeta_4^i}{2})y = x^3 - \prod_{i = 1}^4(f - \tfrac{\zeta_4^i}{2})x^2, $$
for some $ f \in k(t) \setminus k $, with
$$ \Delta = 16f^4\prod_{i = 1}^4(f - \tfrac{\zeta_4^i}{2})^4. $$
There are distinct places $ v_0, \dots, v_5 $ such that
$$ v_0(f) > 0, \qquad v_5(f) < 0, \qquad v_i(f - \tfrac{\zeta_4^i}{2}) > 0, \qquad 1 \le i \le 4. $$
At $ v_0 $,
$$ v_0(c_4) = v_0((\tfrac{5}{4})^2 - \tfrac{3}{2}) = v_0(\tfrac{1}{16}) = 0, \qquad v_0(\Delta) = 4v_0(f). $$
By Lemma \ref{lem:noncyclic}, $ c_{v_0}(E) = 4v_0(f) $, which gives a factor $ 4 \mid c(E) $. At $ v_i $ for $ 1 \le i \le 4 $,
$$ v_i(a_1) = 0, \qquad v_i(a_2) = v_i(a_3) = v_i(f - \tfrac{\zeta_4^i}{2}) > 0, \qquad v_i(\Delta) = 4v_i(f - \tfrac{\zeta_4^i}{2}), $$
so by Lemma \ref{lem:odd}, $ c_{v_i}(E) = 4v_i(f - \tfrac{\zeta_4^i}{2}) $, which give four more factors $ 4 \mid c(E) $. Now the change of variables $ (x, y) \mapsto (x / f^4, y / f^6) $ transforms $ E $ to the Weierstrass equation
$$ y^2 + \dfrac{1}{f^2}xy - \dfrac{\prod_{i = 1}^4(f - \tfrac{\zeta_4^i}{2})}{f^6}y = x^3 - \dfrac{\prod_{i = 1}^4(f - \tfrac{\zeta_4^i}{2})}{f^4}x^2, $$
with
$$ \Delta = \dfrac{16\prod_{i = 1}^4(f - \tfrac{\zeta_4^i}{2})^4}{f^{20}}. $$
At $ v_5 $, Lemma \ref{lem:negative} gives
$$ v_5(c_4) = 2v_5(b_2) = 2v_5(4) = 0, \qquad v_5(\Delta) = -4v_5(f). $$
By Lemma \ref{lem:noncyclic}, $ c_{v_5}(E) = -4v_5(f) $, which gives another factor $ 4 \mid c(E) $.
\end{proof}

\begin{example}
The elliptic curve $ E / \F_5(t) $ given by the parameterisation in Proposition \ref{prop:c4c4} with $ f = t $ has $ c(E) = 4^6 $.
\end{example}

\begin{proposition}
\label{prop:c6c3}
Let $ k $ be a finite field of characteristic $ p \ne 3 $, and let $ E / k(t) $ be a non-isotrivial elliptic curve with $ E(k(t))_{\tors} \cong \C_6 \oplus \C_3 $.
\begin{enumerate}
\item If $ p \ne 2 $, then $ 3^4 \cdot 6^4 $ divides $ c(E) $.
\item If $ p = 2 $, then $ 6^4 $ divides $ c(E) $.
\end{enumerate}
\end{proposition}

\pagebreak

\begin{proof}
By Lemma \ref{lem:noncyclic}, $ \zeta_3 \in k $. By Theorem \ref{thm:torsion}, $ E $ is given by the Weierstrass equation
$$
\begin{aligned}
& y^2 + \dfrac{f_1}{(f - 1)^3}xy - \dfrac{f(f - \zeta_3)(f - \zeta_3^2)(2f + \zeta_3)(2f + \zeta_3^2)}{(f - 1)^6}y \\
& \qquad = x^3 - \dfrac{f(f - \zeta_3)(f - \zeta_3^2)(2f + \zeta_3)(2f + \zeta_3^2)}{(f - 1)^6}x^2,
\end{aligned}
$$
for some $ f \in k(t) \setminus k $, with
$$ \Delta = \dfrac{f^6(f - \zeta_3)^6(f - \zeta_3^2)^6(2f + 1)^3(2f + \zeta_3)^3(2f + \zeta_3^2)^3}{(f - 1)^{30}}, $$
where $ f_1 := 2f^3 - 2f^2 + 4f - 1 $. There are distinct places $ v_0, \dots, v_3 $ such that
$$ v_0(f) > 0, \qquad v_1(f - 1) > 0, \qquad v_2(f - \zeta_3) > 0, \qquad v_3(f - \zeta_3^2) > 0. $$
At $ v_0, v_2, v_3 $,
$$
\begin{aligned}
v_0(a_1) & = 0, & v_0(a_2) & = v_0(a_3) = v_0(f) > 0, & v_0(\Delta) & = 6v_0(f), \\
v_2(a_1) & = 0, & v_2(a_2) & = v_2(a_3) = v_2(f - \zeta_3) > 0, & v_2(\Delta) & = 6v_2(f - \zeta_3), \\
v_3(a_1) & = 0, & v_3(a_2) & = v_3(a_3) = v_3(f - \zeta_3^2) > 0, & v_3(\Delta) & = 6v_3(f - \zeta_3^2).
\end{aligned}
$$
By Lemma \ref{lem:odd},
$$ c_{v_0}(E) = 6v_0(f), \qquad c_{v_2}(E) = 6v_2(f - \zeta_3), \qquad c_{v_3}(E) = 6v_3(f - \zeta_3^2), $$
which give three factors $ 6 \mid c(E) $. Now the change of variables
$$ (x, y) \mapsto \left(\dfrac{(f - 1)^6}{f^6}x, \dfrac{(f - 1)^9}{f^9}y\right) $$
transforms $ E $ to the Weierstrass equation
$$
\begin{aligned}
& y^2 + \dfrac{f_1}{f^3}xy - \dfrac{(f - 1)^3(f - \zeta_3)(f - \zeta_3^2)(2f + \zeta_3)(2f + \zeta_3^2)}{f^8}y \\
& \qquad = x^3 - \dfrac{(f - \zeta_3)(f - \zeta_3^2)(2f + \zeta_3)(2f + \zeta_3^2)}{f^5}x^2,
\end{aligned}
$$
with
$$ \Delta = \dfrac{(f - 1)^6(f - \zeta_3)^6(f - \zeta_3^2)^6(2f + 1)^3(2f + \zeta_3)^3(2f + \zeta_3^2)^3}{f^{30}}. $$
At $ v_1 $,
$$ v_1(c_4) = 2v_1(b_2) = 2v_1(27) = 0, \qquad v_1(\Delta) = 6v_1(f - 1). $$
By Lemma \ref{lem:noncyclic}, $ c_{v_1}(E) = 6v_1(f - 1) $, which gives another factor $ 6 \mid c(E) $.

Let $ p \ne 2 $. There are places $ v_4, \dots, v_7 $, distinct from $ v_0, \dots, v_3 $, such that
$$ v_4(2f + 1) > 0, \qquad v_5(2f + \zeta_3) > 0, \qquad v_6(2f + \zeta_3^2) > 0, \qquad v_7(f) < 0. $$

\pagebreak

\noindent At $ v_5 $ and $ v_6 $,
$$
\begin{aligned}
v_5(a_1) & = 0, & v_5(a_2) & = v_5(a_3) = v_5(2f + \zeta_3) > 0, & v_5(\Delta) & = 3v_5(2f + \zeta_3), \\
v_6(a_1) & = 0, & v_6(a_2) & = v_6(a_3) = v_6(2f + \zeta_3^2) > 0, & v_6(\Delta) & = 3v_6(2f + \zeta_3^2).
\end{aligned}
$$
At $ v_7 $, Lemma \ref{lem:negative} gives
$$ v_7(a_1) = 0, \qquad v_7(a_2) = v_7(a_3) = -v_7(f) > 0, \qquad v_7(\Delta) = -3v_7(f). $$
By Lemma \ref{lem:odd},
$$ c_{v_5}(E) = 3v_5(2f + \zeta_3), \qquad c_{v_6}(E) = 3v_6(2f + \zeta_3^2), \qquad c_{v_7}(E) = -3v_7(f), $$
which give three factors $ 3 \mid c(E) $. At $ v_4 $,
$$ v_4(c_4) = v_4(2^4 \cdot 3^6) = 0, \qquad v_4(\Delta) = 3v_4(2f + 1). $$
By Lemma \ref{lem:noncyclic}, $ c_{v_4}(E) = 3v_4(2f + 1) $, which gives another factor $ 3 \mid c(E) $.
\end{proof}

\begin{example}
The elliptic curve $ E / \F_7(t) $ given by the parameterisation in Proposition \ref{prop:c6c3} with $ f = t $ has $ c(E) = 3^4 \cdot 6^4 $. If $ \F_4 := \F_2[\alpha] / (\alpha^2 + \alpha + 1) $, then the elliptic curve $ E / \F_4(t) $ given by the same parameterisation with $ f = t $ has $ c(E) = 6^4 $.
\end{example}

\begin{proposition}
\label{prop:c5c5}
Let $ k $ be a finite field of characteristic $ p \ne 5 $, and let $ E / k(t) $ be a non-isotrivial elliptic curve with $ E(k(t))_{\tors} \cong \C_5 \oplus \C_5 $. Then $ 5^{12} $ divides $ c(E) $.
\end{proposition}

\begin{proof}
By Lemma \ref{lem:noncyclic}, $ \zeta_5 \in k $. By Theorem \ref{thm:torsion}, $ E $ is given by the Weierstrass equation
$$
\begin{aligned}
& y^2 + \dfrac{f_1}{f\prod_{i = 1}^4(f - 1 - \zeta_5^i)}xy - \dfrac{\prod_{i = 1}^4(f - \zeta_5^i - \zeta_5^{2i})}{f\prod_{i = 1}^4(f - 1 - \zeta_5^i)}y \\
& \qquad = x^3 - \dfrac{\prod_{i = 1}^4(f - \zeta_5^i - \zeta_5^{2i})}{f\prod_{i = 1}^4(f - 1 - \zeta_5^i)}x^2,
\end{aligned}
$$
for some $ f \in k(t) \setminus k $, with
$$ \Delta = -\dfrac{\prod_{1 \le i < j \le 4}(f - \zeta_5^i - \zeta_5^j)^5}{f^7\prod_{i = 1}^4(f - 1 - \zeta_5^i)^7}, \qquad f_1 := f^5 - 4f^4 + 2f^3 - 6f^2 - 2f - 1. $$
There are distinct places $ v_0, \dots, v_5 $ and $ v_{i, j} $ such that
$$
\begin{aligned}
v_0(f) & > 0, & v_i(f - 1 - \zeta_5^i) & > 0, & 1 \le i \le 4, \\
v_5(f) & < 0, & v_{i, j}(f - \zeta_5^i - \zeta_5^j) & > 0, & 1 \le i < j \le 4.
\end{aligned}
$$
At $ v_5 $, Lemma \ref{lem:negative} gives
$$ v_5(a_1) = 0, \qquad v_5(a_2) = v_5(a_3) = -v_5(f) > 0, \qquad v_5(\Delta) = -5v_5(f). $$

\pagebreak

\noindent At $ v_{i, j} $ for $ 1 \le i < j \le 4 $ with $ i + j \ne 5 $,
$$
\begin{aligned}
v_{i, j}(a_1) & = 0, & v_{i, j}(a_2) & = v_{i, j}(a_3) = v_{i, j}(f - \zeta_5^i - \zeta_5^j) > 0, \\
& & v_{i, j}(\Delta) & = 5v_{i, j}(f - \zeta_5^i - \zeta_5^j).
\end{aligned}
$$
By Lemma \ref{lem:odd}, $ c_{v_5}(E) = -5v_5(f) $ and
$$ c_{v_{i, j}}(E) = 5v_{i, j}(f - \zeta_5^i - \zeta_5^j), \qquad 1 \le i < j \le 4, \qquad i + j \ne 5, $$
which give five factors $ 5 \mid c(E) $. At $ v_{i, j} $ for $ 1 \le i < j \le 4 $ with $ i + j = 5 $,
$$ v_{i, j}(c_4) = v_{i, j}(\tfrac{5}{(5\zeta_5^i + 5\zeta_5^j - 3)^2}) = 0, \qquad v_{i, j}(\Delta) = 5v_{i, j}(f - \zeta_5^i - \zeta_5^j), $$
so by Lemma \ref{lem:noncyclic}, $ c_{v_{i, j}}(E) = 5v_{i, j}(f - \zeta_5^i - \zeta_5^j) $, which give two more factors $ 5 \mid c(E) $. Now the change of variables
$$ (x, y) \mapsto \left(f^2x\prod_{i = 1}^4(f - 1 - \zeta_5^i)^2, f^3y\prod_{i = 1}^4(f - 1 - \zeta_5^i)^3\right) $$
transforms $ E $ to the Weierstrass equation
$$
\begin{aligned}
& y^2 + f_1xy - f^2\prod_{i = 1}^4(f - 1 - \zeta_5^i)^2(f - \zeta_5^i - \zeta_5^{2i})y \\
& \qquad = x^3 - f\prod_{i = 1}^4(f - 1 - \zeta_5^i)(f - \zeta_5^i - \zeta_5^{2i})x^2,
\end{aligned}
$$
with
$$ \Delta = -f^5\prod_{i = 1}^4(f - 1 - \zeta_5^i)^5\prod_{1 \le i < j \le 4}(f - \zeta_5^i - \zeta_5^j)^5. $$
At $ v_0 $,
$$ v_0(a_1) = 0, \qquad
\begin{aligned}
v_0(a_2) & = v_0(f) > 0, \\
v_0(a_3) & = 2v_0(f) > 0,
\end{aligned}
\qquad v_0(\Delta) = 5v_0(f). $$
At $ v_i $ for $ 1 \le i \le 4 $,
$$ v_i(a_1) = 0, \qquad
\begin{aligned}
v_i(a_2) & = v_i(f - 1 - \zeta_5^i) > 0, \\
v_i(a_3) & = 2v_i(f - 1 - \zeta_5^i) > 0,
\end{aligned}
\qquad v_i(\Delta) = 5v_i(f - 1 - \zeta_5^i). $$
By Lemma \ref{lem:odd}, $ c_{v_0}(E) = 5v_0(f) $ and
$$ c_{v_i}(E) = 5v_i(f - 1 - \zeta_5^i), \qquad 1 \le i \le 4, $$
which give five more factors $ 5 \mid c(E) $.
\end{proof}

\begin{example}
The elliptic curve $ E / \F_{11}(t) $ given by the parameterisation in Proposition \ref{prop:c5c5} with $ f = t $ has $ c(E) = 5^{12} $.
\end{example}

\pagebreak

\begin{proposition}
\label{prop:c10c2}
Let $ k $ be a finite field of characteristic $ 5 $, and let $ E / k(t) $ be a non-isotrivial elliptic curve with $ E(k(t))_{\tors} \cong \C_{10} \oplus \C_2 $. Then $ 10^6 $ divides $ c(E) $.
\end{proposition}

\begin{proof}
By Theorem \ref{thm:torsion}, $ E $ is given by the Weierstrass equation
$$
\begin{aligned}
& y^2 + \dfrac{4(f^2 + 4f + 2)(f^4 + 3f^3 + f + 2)}{f_1^2}xy \\
& \quad - \dfrac{f(f + 1)^3(f + 2)^2(f + 3)^3(f + 4)}{4f_1^4}y \\
& \qquad = x^3 - \dfrac{f(f + 1)^3(f + 2)^2(f + 3)^3(f + 4)}{4f_1^4}x^2,
\end{aligned}
$$
for some $ f \in k(t) \setminus k $, with
$$ \Delta = \dfrac{4\prod_{i = 0}^4(f + i)^{10}}{f_1^{18}}, \qquad f_1 := f^2 + 4f + 1. $$
There are distinct places $ v_0, \dots, v_5 $ such that
$$ v_5(f) < 0, \qquad v_i(f + i) > 0, \qquad 0 \le i \le 4. $$
At $ v_i $ for $ 0 \le i \le 4 $,
$$ v_i(a_1) = 0, \qquad v_i(a_2) = v_i(a_3) \ge v_i(f + i) > 0, \qquad v_i(\Delta) = 10v_i(f + i), $$
so by Lemma \ref{lem:odd}, $ c_{v_i}(E) = 10v_i(f + i) $, which give five factors $ 10 \mid c(E) $. Now the change of variables
$$ (x, y) \mapsto \left(\dfrac{f_1^4}{f^{12}}x, \dfrac{f_1^6}{f^{18}}y\right) $$
transforms $ E $ to the Weierstrass equation
$$
\begin{aligned}
& y^2 + \dfrac{(3f^2 + 2f + 1)(3f^4 + 4f^3 + 3f + 1)}{f^6}xy \\
& \quad - \dfrac{(f + 1)^3(f + 2)^2(f + 3)^3(f + 4)f_1^2}{4f^{17}}y \\
& \qquad = x^3 - \dfrac{(f + 1)^3(f + 2)^2(f + 3)^3(f + 4)}{4f^{11}}x^2,
\end{aligned}
$$
with
$$ \Delta = \dfrac{4\prod_{i = 1}^4(f + i)^{10}f_1^6}{f^{62}}. $$
At $ v_5 $, Lemma \ref{lem:negative} gives
$$ v_5(a_1) = 0, \qquad
\begin{aligned}
v_5(a_2) & = -2v_5(f) > 0, \\
v_5(a_3) & = -4v_5(f) > 0,
\end{aligned}
\qquad v_5(\Delta) = -10v_5(f). $$
By Lemma \ref{lem:odd}, $ c_{v_5}(E) = -10v_5(f) $, which gives another factor $ 10 \mid c(E) $.
\end{proof}

\pagebreak

\begin{example}
If $ \F_{25} := \F_5[\alpha] / (\alpha^2 + 4\alpha + 2) $, then the elliptic curve $ E / \F_{25}(t) $ given by the parameterisation in Proposition \ref{prop:c10c2} with $ f = ((2\alpha + 2)t^2 + \alpha) / (t^2 + 4\alpha + 1) $ has $ c(E) = 10^6 $.
\end{example}

\begin{proposition}
\label{prop:c10c5}
Let $ k $ be a finite field of characteristic $ 2 $, and let $ E / k(t) $ be a non-isotrivial elliptic curve with $ E(k(t))_{\tors} \cong \C_{10} \oplus \C_5 $. Then $ 10^{12} $ divides $ c(E) $.
\end{proposition}

\begin{proof}
By Lemma \ref{lem:noncyclic}, $ \zeta_5 \in k $. By Theorem \ref{thm:torsion}, $ E $ is given by the Weierstrass equation
$$
\begin{aligned}
& y^2 + \dfrac{\prod_{i = 0}^4(f + \zeta_5^i)^2}{\prod_{i = 1}^2(f + \zeta_5^i + \zeta_5^{-i})^5}xy - \dfrac{f^3\prod_{i = 1}^4(f + 1 + \zeta_5^i)^3(f + \zeta_5^i + \zeta_5^{2i})}{\prod_{i = 1}^2(f + \zeta_5^i + \zeta_5^{-i})^{10}}y \\
& \qquad = x^3 - \dfrac{f^3\prod_{i = 1}^4(f + 1 + \zeta_5^i)^3(f + \zeta_5^i + \zeta_5^{2i})}{\prod_{i = 1}^2(f + \zeta_5^i + \zeta_5^{-i})^{10}}x^2,
\end{aligned}
$$
for some $ f \in k(t) \setminus k $, with
$$ \Delta = \dfrac{f^{10}\prod_{i = 1}^4(f + 1 + \zeta_5^i)^{10}(f + \zeta_5^i + \zeta_5^{2i})^{10}}{\prod_{i = 1}^2(f + \zeta_5^i + \zeta_5^{-i})^{50}}. $$
There are distinct places $ v_0, \dots, v_5 $ and $ v_{i, j} $ such that
$$
\begin{aligned}
v_0(f) & > 0, & v_i(f + 1 + \zeta_5^i) & > 0, & 1 \le i \le 4, \\
v_5(f) & < 0, & v_{i, j}(f + \zeta_5^i + \zeta_5^j) & > 0, & 1 \le i < j \le 4.
\end{aligned}
$$
At $ v_0 $,
$$ v_0(a_1) = 0, \qquad v_0(a_2) = v_0(a_3) = 3v_0(f) > 0, \qquad v_0(\Delta) = 10v_0(f). $$
At $ v_5 $, Lemma \ref{lem:negative} gives
$$ v_5(a_1) = 0, \qquad v_5(a_2) = v_5(a_3) = -v_5(f) > 0, \qquad v_5(\Delta) = -10v_5(f). $$
At $ v_i $ for $ 1 \le i \le 4 $,
$$
\begin{aligned}
v_i(a_1) & = 0, & v_i(a_2) & = v_i(a_3) = 3v_i(f + 1 + \zeta_5^i) > 0, \\
& & v_i(\Delta) & = 10v_i(f + 1 + \zeta_5^i).
\end{aligned}
$$
At $ v_{i, j} $ for $ 1 \le i < j \le 4 $ with $ i + j \ne 5 $,
$$
\begin{aligned}
v_{i, j}(a_1) & = 0, & v_{i, j}(a_2) & = v_{i, j}(a_3) = v_{i, j}(f + \zeta_5^i + \zeta_5^j) > 0, \\
& & v_{i, j}(\Delta) & = 10v_{i, j}(f + \zeta_5^i + \zeta_5^j).
\end{aligned}
$$
By Lemma \ref{lem:odd},
$$
\begin{aligned}
c_{v_0}(E) & = 10v_0(f), & c_{v_5}(E) & = -10v_5(f), \\
c_{v_i}(E) & = 10v_i(f + 1 + \zeta_5^i), & & 1 \le i \le 4, \\
c_{v_{i, j}}(E) & = 10v_{i, j}(f + \zeta_5^i + \zeta_5^j), & & 1 \le i < j \le 4, \ i + j \ne 5,
\end{aligned}
$$

\pagebreak

\noindent which give ten factors $ 10 \mid c(E) $. Now the change of variables
$$ (x, y) \mapsto \left(x\prod_{i = 1}^2(f + \zeta_5^i + \zeta_5^{-i})^{10}, y\prod_{i = 1}^2(f + \zeta_5^i + \zeta_5^{-i})^{15}\right) $$
transforms $ E $ to the Weierstrass equation
$$
\begin{aligned}
& y^2 + \prod_{i = 0}^4(f + \zeta_5^i)^2xy - f^3\prod_{i = 1}^4(f + 1 + \zeta_5^i)^3(f + \zeta_5^i + \zeta_5^{2i})\prod_{i = 1}^2(f + \zeta_5^i + \zeta_5^{-i})^5y \\
& \qquad = x^3 - f^3\prod_{i = 1}^4(f + 1 + \zeta_5^i)^3(f + \zeta_5^i + \zeta_5^{2i})x^2,
\end{aligned}
$$
with
$$ \Delta = f^{10}\prod_{i = 1}^4(f + 1 + \zeta_5^i)^{10}\prod_{1 \le i < j \le 4}(f + \zeta_5^i + \zeta_5^j)^{10}. $$
At $ v_{i, j} $ for $ 1 \le i < j \le 4 $ with $ i + j = 5 $,
$$ v_{i, j}(c_4) = 4v_{i, j}(a_1) = v_{i, j}((\zeta_5^i + \zeta_5^j)^8) = 0, \qquad v_{i, j}(\Delta) = 10v_{i, j}(f + \zeta_5^i + \zeta_5^j), $$
so by Lemma \ref{lem:noncyclic}, $ c_{v_{i, j}}(E) = 10v_{i, j}(f + \zeta_5^i + \zeta_5^j) $, which give two more factors $ 10 \mid c(E) $.
\end{proof}

\begin{example}
If $ \F_{16} := \F_2[\alpha] / (\alpha^4 + \alpha + 1) $, then the elliptic curve $ E / \F_{16}(t) $ given by the parameterisation in Proposition \ref{prop:c10c5} with $ f = t $ has $ c(E) = 10^{12} $.
\end{example}

\begin{proposition}
\label{prop:c12c2}
Let $ k $ be a finite field of characteristic $ p $, and let $ E / k(t) $ be a non-isotrivial elliptic curve with $ E(k(t))_{\tors} \cong \C_{12} \oplus \C_2 $. Then $ 6^2 \cdot 12^4 $ divides $ c(E) $.
\end{proposition}

\begin{proof}
By Theorem \ref{thm:torsion}, $ p = 3 $. By Proposition \ref{prop:mcdonald}, $ E $ is given by the Weierstrass equation
$$
\begin{aligned}
& y^2 + \left(1 - \dfrac{f(f + 1)(f^2 + 1)f_2}{(f + 2)^3}\right)xy - \dfrac{f(f + 1)(f^2 + 1)f_1^2}{(f + 2)^4}y \\
& \qquad = x^3 - \dfrac{f(f + 1)(f^2 + 1)f_1^2}{(f + 2)^4}x^2,
\end{aligned}
$$
for some $ f \in k(t) \setminus k $, with
$$ \Delta = \dfrac{f^{12}(f + 1)^{12}(f^2 + 1)^6f_1^6}{(f + 2)^{24}}, $$
where $ f_1 := f^2 + f + 2 $ and $ f_2 := f^2 + 2f + 2 $. Since $ f^2 + 1, f_1 \notin k $ by Lemma \ref{lem:negative}, there are distinct places $ v_0, \dots, v_5 $ such that
$$
\begin{aligned}
v_0(f) & > 0, & v_1(f + 1) & > 0, & v_2(f + 2) & > 0, \\
v_3(f^2 + 1) & > 0, & v_4(f_1) & > 0, & v_5(f) & < 0.
\end{aligned}
$$

\pagebreak

\noindent At $ v_0 $ and $ v_1 $,
$$
\begin{aligned}
v_0(a_1) & = 0, & v_0(a_2) & = v_0(a_3) = v_0(f) > 0, & v_0(\Delta) & = 12v_0(f), \\
v_1(a_1) & = 0, & v_1(a_2) & = v_1(a_3) = v_1(f + 1) > 0, & v_1(\Delta) & = 12v_1(f + 1).
\end{aligned}
$$
At $ v_3 $ and $ v_4 $, Lemma \ref{lem:identity} \footnote{$ f_1 = (f^2 + 1) + (f + 1) $, $ f^2 + 1 = f_1 - (f + 1) $, and $ (f + 2)^3a_1 = -(f^4 + 2f^3 + f^2 + 2f)f_1 - (f + 1) $} gives
$$
\begin{aligned}
& v_3(f) = v_3(f + 1) = v_3(f + 2) = v_3(f_1) = 0, \\
& v_4(f) = v_4(f + 1) = v_4(f + 2) = v_4(f^2 + 1) = v_4((f + 2)^3a_1) = 0,
\end{aligned}
$$
so
$$
\begin{aligned}
v_3(a_1) & = 0, & v_3(a_2) & = v_3(a_3) = v_3(f^2 + 1) > 0, & v_3(\Delta) & = 6v_3(f^2 + 1), \\
v_4(a_1) & = 0, & v_4(a_2) & = v_4(a_3) = 2v_4(f_1) > 0, & v_4(\Delta) & = 6v_4(f_1).
\end{aligned}
$$
By Lemma \ref{lem:odd},
$$
\begin{aligned}
c_{v_0}(E) & = 12v_0(f), & c_{v_1}(E) & = 12v_1(f + 1), \\
c_{v_3}(E) & = 6v_3(f^2 + 1), & c_{v_4}(E) & = 6v_4(f_1),
\end{aligned}
$$
which give two factors $ 12 \mid c(E) $ and two factors $ 6 \mid c(E) $. Now the change of variables
$$ (x, y) \mapsto \left(\dfrac{(f + 2)^6}{(f^2 + 1)^6}x, \dfrac{(f + 2)^9}{(f^2 + 1)^9}y\right) $$
transforms $ E $ to the Weierstrass equation
$$
\begin{aligned}
& y^2 + \dfrac{(f + 2)^3 - f(f + 1)(f^2 + 1)f_2}{(f^2 + 1)^3}xy - \dfrac{f(f + 1)(f + 2)^5f_1^2}{(f^2 + 1)^8}y \\
& \qquad = x^3 - \dfrac{f(f + 1)(f + 2)^2f_1^2}{(f^2 + 1)^5}x^2,
\end{aligned}
$$
with
$$ \Delta = \dfrac{f^{12}(f + 1)^{12}(f + 2)^{12}f_1^6}{(f^2 + 1)^{30}}. $$
At $ v_2 $,
$$ v_2(a_1) = 0, \qquad
\begin{aligned}
v_2(a_2) & = 2v_2(f + 2) > 0, \\
v_2(a_3) & = 5v_2(f + 2) > 0,
\end{aligned}
\qquad v_2(\Delta) = 12v_2(f + 2). $$
At $ v_5 $, Lemma \ref{lem:negative} gives
$$ v_5(a_1) = 0, \qquad
\begin{aligned}
v_5(a_2) & = -2v_5(f) > 0, \\
v_5(a_3) & = -5v_5(f) > 0,
\end{aligned}
\qquad v_5(\Delta) = -12v_5(f). $$
By Lemma \ref{lem:odd}, $ c_{v_2}(E) = 12v_2(f + 2) $ and $ c_{v_5}(E) = -12v_5(f) $, which give two more factors $ 12 \mid c(E) $.
\end{proof}

\begin{example}
The elliptic curve $ E / \F_3(t) $ given by the parameterisation in Proposition \ref{prop:c12c2} with $ f = t $ has $ c(E) = 6^2 \cdot 12^4 $.
\end{example}

\pagebreak

\section{Exceptional torsion}
\label{sec:exceptional}

\begin{lemma}
\label{lem:exceptional}
Let $ F $ be a complete discretely valued field with finite residue field and normalised valuation $ v $, and let $ E / F $ be an elliptic curve given by an integral Weierstrass equation with
$$ v(a_1) \ge 1, \qquad v(a_2) = 1, \qquad v(a_3) \ge 2, \qquad v(a_4) \ge 3, \qquad v(a_6) \ge 4. $$
Then the equation is minimal with additive reduction of type $ \I_n^* $ for some $ n > 0 $, and $ c_v(E) $ is even.
\end{lemma}

\begin{proof}
Let $ \pi $ be a uniformiser of $ F $. The valuation conditions give
$$ v(b_6) \ge 4, \qquad v(b_8) \ge 5, \qquad v(c_4) \ge 2, \qquad v(\Delta) \ge 7, $$
so Tate's algorithm reaches step $ 6 $. Since the polynomial $ T^3 + (a_2 / \pi)T^2 + (a_4 / \pi^2)T + a_6 / \pi^3 $ splits to $ T^2(T + a_2 / \pi) $ in the residue field, Tate's algorithm terminates at step $ 7 $, so the reduction is additive of type $ \I_n^* $ for some $ n > 0 $ and $ c_v(E) $ is even \cite[Page 367]{Sil94}.
\end{proof}

\begin{proposition}
\label{prop:c2}
Let $ k $ be a finite field of characteristic $ p $, and let $ E / k(t) $ be a non-isotrivial elliptic curve with $ E(k(t))_{\tors} \cong \C_2 $. Then $ 2 $ divides $ c(E) $.
\end{proposition}

\begin{proof}
Let $ p \ne 2 $. By Theorem \ref{thm:torsion}, $ E $ is given by the Weierstrass equation
$$ y^2 = x^3 + Ax^2 + Bx, \qquad A, B \in k(t). $$
Then $ A \ne 0 $ and $ B / A^2 \notin k $, since
$$ j = \dfrac{256(A^2 - 3B)^3}{B^2(A^2 - 4B)} = \dfrac{256(3B / A^2 - 1)^3}{(B / A^2)^2(4B / A^2 - 1)} \notin k, $$
so there is a place $ v $ such that $ v(B) > 2v(A) $. Let $ \pi $ be a uniformiser at $ v $, and let $ n := \lfloor v(A) / 2\rfloor $, so that $ v(A) - 2n \in \{0, 1\} $. Now the change of variables $ (x, y) \mapsto (x / \pi^{2n}, y / \pi^{3n}) $ transforms $ E $ to the Weierstrass equation
$$ y^2 = x^3 + \dfrac{A}{\pi^{2n}}x^2 + \dfrac{B}{\pi^{4n}}x, $$
with
$$ c_4 = \dfrac{16(A^2 - 3B)}{\pi^{4n}}, \qquad \Delta = \dfrac{16B^2(A^2 - 4B)}{\pi^{12n}}. $$
If $ v(A) = 2n $, then
$$ v(c_4) = 2v(A) - 4n = 0, \qquad v(\Delta) = 2(v(A) + v(B) - 6n), $$
so Lemma \ref{lem:even} gives $ 2 \mid c_v(E) $, and hence a factor $ 2 \mid c(E) $. If $ v(A) = 2n + 1 $, then
$$ v(a_2) = v(A) - 2n = 1, \qquad v(a_4) = v(B) - 4n > 2v(A) - 4n = 2, $$
so Lemma \ref{lem:exceptional} gives $ 2 \mid c_v(E) $, and hence a factor $ 2 \mid c(E) $.

\pagebreak

Now let $ p = 2 $. By Theorem \ref{thm:torsion}, $ E $ is given by the Weierstrass equation
$$ y^2 + xy = x^3 + Ax^2 + Bx, \qquad A, B \in k(t), $$
with $ c_4 = 1 $ and $ \Delta = B^2 $. Then $ B \notin k $, since $ j = 1 / B^2 \notin k $, so there is a place $ v $ such that $ v(B) > 0 $. Choose $ S \in k(t) $ to maximise
$$ m := \min(v(S^2 + S + A), 0) \le 0, $$
let $ n := \lfloor m / 2\rfloor \le 0 $, and let $ \pi $ be a uniformiser at $ v $. The change of variables $ (x, y) \mapsto (x / \pi^{2n}, (y + Sx) / \pi^{3n}) $ transforms $ E $ to the Weierstrass equation
$$ y^2 + \dfrac{1}{\pi^n}xy = x^3 + \dfrac{S^2 + S + A}{\pi^{2n}}x^2 + \dfrac{B}{\pi^{4n}}x, $$
with $ c_4 = 1 / \pi^{4n} $ and $ \Delta = B^2 / \pi^{12n} $. If $ n = 0 $, then
$$ v(c_4) = 0, \qquad v(\Delta) = 2v(B), $$
so Lemma \ref{lem:even} gives $ 2 \mid c_v(E) $, and hence a factor $ 2 \mid c(E) $. If $ n \le -1 $, then $ m < 0 $ is odd by the maximality of $ m $, and
$$ v(a_1) = -n \ge 1, \qquad v(a_2) = m - 2n = 1, \qquad v(a_4) = v(B) - 4n \ge 5, $$
so Lemma \ref{lem:exceptional} gives $ 2 \mid c_v(E) $, and hence a factor $ 2 \mid c(E) $.
\end{proof}

\begin{example}
The elliptic curve $ E / \F_3(t) $ given by its parameterisation in Proposition \ref{prop:c2} with $ A = 2t^2 $ and $ B = 1 $ has $ c(E) = 2 $. The elliptic curve $ E / \F_2(t) $ given by its parameterisation in Proposition \ref{prop:c2} with $ A = 0 $ and $ B = t^4 + t^3 + 1 $ also has $ c(E) = 2 $.
\end{example}

\begin{proposition}
\label{prop:c4}
Let $ k $ be a finite field of characteristic $ p $, and let $ E / k(t) $ be a non-isotrivial elliptic curve with $ E(k(t))_{\tors} \cong \C_4 $.
\begin{enumerate}
\item If $ p \ne 2 $, then $ 2^3 $ divides $ c(E) $.
\item If $ p = 2 $, then $ 2^2 $ divides $ c(E) $.
\end{enumerate}
\end{proposition}

\begin{proof}
By Theorem \ref{thm:torsion}, $ E $ is given by the Weierstrass equation
$$ y^2 + xy - fy = x^3 - fx^2, $$
for some $ f \in k(t) \setminus k $, with $ \Delta = f^4(16f + 1) $. Since $ f \notin k $, there is a place $ v_1 $ such that $ v_1(f) > 0 $. At $ v_1 $,
$$ v_1(a_1) = 0, \qquad v_1(a_2) = v_1(a_3) = v_1(f) > 0, \qquad v_1(\Delta) = 4v_1(f). $$
By Lemma \ref{lem:odd}, $ c_{v_1}(E) = 4v_1(f) $, which gives two factors $ 2 \mid c(E) $.

Let $ p \ne 2 $. There is a place $ v_2 $, distinct from $ v_1 $, such that $ v_2(f) < 0 $. Let $ \pi $ be a uniformiser at $ v_2 $, and let $ n := \lceil -v_2(f) / 2\rceil $, so that $ 2n + v_2(f) \in \{0, 1\} $. Now the change of variables $ (x, y) \mapsto (\pi^{2n}x, \pi^{3n}y) $ transforms $ E $ to the Weierstrass equation
$$ y^2 + \pi^nxy - \pi^{3n}fy = x^3 - \pi^{2n}fx^2, $$
with
$$ c_4 = \pi^{4n}(16f^2 + 16f + 1), \qquad \Delta = \pi^{12n}f^4(16f + 1). $$

\pagebreak

\noindent At $ v_2 $, if $ v_2(f) = -2n $, then Lemma \ref{lem:negative} gives
$$ v_2(c_4) = 4n + 2v_2(f) = 0, \qquad v_2(\Delta) = 12n + 5v_2(f) = 2n, $$
so Lemma \ref{lem:even} gives $ 2 \mid c_{v_2}(E) $, and hence another factor $ 2 \mid c(E) $. If $ v_2(f) = 1 - 2n $, then
$$ v_2(a_1) = n \ge 1, \qquad v_2(a_2) = 2n + v_2(f) = 1, \qquad v_2(a_3) = 3n + v_2(f) \ge 2, $$
so Lemma \ref{lem:exceptional} gives $ 2 \mid c_{v_2}(E) $, and hence another factor $ 2 \mid c(E) $.
\end{proof}

\begin{example}
The elliptic curve $ E / \F_3(t) $ given by the parameterisation in Proposition \ref{prop:c4} with $ f = t^2 + 1 $ has $ c(E) = 2^3 $. The elliptic curve $ E / \F_2(t) $ given by the same parameterisation with $ f = t^8 + t^6 + t^5 + t^4 + 1 $ has $ c(E) = 2^2 $.
\end{example}

\begin{proposition}
\label{prop:c2c2}
Let $ k $ be a finite field of characteristic $ p \ne 2 $, and let $ E / k(t) $ be a non-isotrivial elliptic curve with $ E(k(t))_{\tors} \cong \C_2 \oplus \C_2 $. Then $ 2^3 $ divides $ c(E) $.
\end{proposition}

\begin{proof}
By Theorem \ref{thm:torsion}, $ E $ is given by the Weierstrass equation
$$ y^2 = x(x - A)(x - B), \qquad A, B \in k(t). $$
Then $ A, B \ne 0 $ and $ A / B, B / A \notin k $, since
$$ j = \dfrac{256(A^2 - AB + B^2)^3}{A^2B^2(A - B)^2} \notin k, $$
so there are distinct places $ v_1, v_2, v_3 $ such that
$$ v_1(A) > v_1(B), \qquad v_2(A) < v_2(B), \qquad v_3(A - B) > v_3(B). $$
Let $ v $ be any of these places, and let $ \pi $ be a uniformiser at $ v $. Let $ m := \min(v(A), v(B)) $, and let $ n := \lfloor m / 2\rfloor $, so that $ m - 2n \in \{0, 1\} $. Now the change of variables
$$ (x, y) \mapsto \left(\dfrac{x - \tfrac{2AB}{A + B}}{\pi^{2n}}, \dfrac{y}{\pi^{3n}}\right) $$
transforms $ E $ to the Weierstrass equation
$$ y^2 = x^3 - \dfrac{A^2 - 4AB + B^2}{\pi^{2n}(A + B)}x^2 - \dfrac{3AB(A - B)^2}{\pi^{4n}(A + B)^2}x - \dfrac{2A^2B^2(A - B)^2}{\pi^{6n}(A + B)^3}, $$
with
$$ c_4 = \dfrac{16(A^2 - AB + B^2)}{\pi^{4n}}, \qquad \Delta = \dfrac{16A^2B^2(A - B)^2}{\pi^{12n}}. $$
If $ m = 2n $, then
$$
\begin{aligned}
v(c_4) & = v(A^2 - AB + B^2) - 4n = 2m - 4n = 0, \\
v(\Delta) & = 2(v(A) + v(B) + v(A - B) - 6n),
\end{aligned}
$$

\pagebreak

\noindent so Lemma \ref{lem:even} gives $ 2 \mid c_v(E) $. If $ m = 2n + 1 $, then
$$
\begin{aligned}
v(a_2) & = v(A^2 - 4AB + B^2) - 2n - m = m - 2n = 1, \\
v(a_4) & \ge v(A) + v(B) + 2v(A - B) - 4n - 2m \ge 2m - 4n + 1 = 3, \\
v(a_6) & = 2v(A) + 2v(B) + 2v(A - B) - 6n - 3m \ge 3m - 6n + 2 = 5,
\end{aligned}
$$
so Lemma \ref{lem:exceptional} gives $ 2 \mid c_v(E) $. These give three factors $ 2 \mid c(E) $.
\end{proof}

\begin{example}
The elliptic curve $ E / \F_7(t) $ given by the parameterisation in Proposition \ref{prop:c2c2} with $ A = 1 $ and $ B = t^2 + 2 $ has $ c(E) = 2^3 $.
\end{example}

\section{Proofs of main results}
\label{sec:proofs}

\begin{proof}[Proof of Theorem \ref{thm:tamagawa}]
Theorem \ref{thm:torsion} gives a complete list of possible tuples $ (E(k(t))_{\tors}, p) $. For each tuple, there is a corresponding proposition in Sections \ref{sec:odd} to \ref{sec:exceptional}, which shows its corresponding integer $ c $ divides $ c(E) $, and a corresponding example that witnesses sharpness.
\end{proof}

To prove Theorem \ref{thm:general} we will use the raw form polynomials $ F_N(X, Y) \in \Z[X, Y] $ defining the modular curve $ Y_1(N) $ for $ N > 5 $ \cite{Sut}.

\begin{lemma}
\label{lem:modular}
Let $ 5 < N \le 101 $ be a prime, let $ F $ be a discretely valued field with valuation $ v $, and let $ r, s \in F $ such that $ r \ne 0 $. Then $ F_N(X, s) \in \Z[s][X] $ is monic and $ F_N(r, Y) \in \Z[r][Y] $ has leading coefficient $ \pm r^d $ for some $ d \ge 0 $. Furthermore, if $ F_N(r, s) = 0 $, then $ v(s) \ge 0 $ implies $ v(r) \ge 0 $, and $ v(r) = 0 $ implies $ v(s) \ge 0 $.
\end{lemma}

\begin{proof}
By inspecting Sutherland's database \cite{Sut}, both $ F_N(X, s) \in \Z[s][X] $ and $ F_N(r, Y) \in \Z[r][Y] $ have the required leading coefficients. Now assume that $ F_N(r, s) = 0 $. If $ v(s) \ge 0 $ but $ v(r) < 0 $, then every coefficient of $ F_N(X, s) \in \Z[s][X] $ has a non-negative valuation, so Lemma \ref{lem:negative} gives
$$ v(F_N(r, s)) = v(r)\deg_X F_N(X, s) < 0. $$
Similarly, if $ v(r) = 0 $ but $ v(s) < 0 $, then $ v(\pm r^d) = 0 $ and every coefficient of $ F_N(r, Y) \in \Z[r][Y] $ has a non-negative valuation, so Lemma \ref{lem:negative} gives
$$ v(F_N(r, s)) = v(s)\deg_Y F_N(r, Y) < 0. $$
Both of these contradict $ F_N(r, s) = 0 $.
\end{proof}

\begin{proof}[Proof of Theorem \ref{thm:general}]
Let $ N > 5 $. A change of variables translating $ P $ to $ (0, 0) $ transforms $ E $ to the Weierstrass equation in the Tate normal form $ \TT(s(r - 1), rs(r - 1)) $ for some $ r, s \in K $ such that $ F_N(r, s) = 0 $ \cite[Section 2]{Sut12}. If $ s \in k $, then Lemma \ref{lem:modular} shows that $ r $ is algebraic over $ k $, so $ r \in k $ \cite[Lemma 33.8.6]{Stacks}. This contradicts the non-isotriviality of $ E $, so $ s \notin k $, and hence there are distinct places $ v_1 $ and $ v_2 $ such that $ v_1(s) > 0 $ and $ v_2(s) < 0 $. Since $ v_1(s) > 0 $, Lemma \ref{lem:modular} gives $ v_1(r) \ge 0 $. At $ v_1 $,
$$ v_1(a_1) = 0, \qquad v_1(a_2) = v_1(a_3) = v_1(r) + v_1(s) + v_1(r - 1) > 0. $$

\pagebreak

\noindent Since $ v_2(s) < 0 $, Lemma \ref{lem:modular} gives $ v_2(r) \ne 0 $. Let $ \pi $ be a uniformiser at $ v_2 $, let $ m := v_2(r - 1) = \min(v_2(r), 0) \le 0 $, and let $ n := m + v_2(s) < 0 $. Now the change of variables $ (x, y) \mapsto (x / \pi^{2n}, y / \pi^{3n}) $ preserves $ P = (0, 0) $ and transforms $ E $ to the Weierstrass equation
$$ y^2 + \dfrac{1 - s(r - 1)}{\pi^n}xy - \dfrac{rs(r - 1)}{\pi^{3n}}y = x^3 - \dfrac{rs(r - 1)}{\pi^{2n}}x^2. $$
At $ v_2 $,
$$ v_2(a_1) = 0, \qquad v_2(a_2) = v_2(r) - n > 0, \qquad v_2(a_3) = v_2(r) - 2n > 0. $$
By Lemma \ref{lem:odd}, both equations are minimal. In both places $ v $, the rational point $ P = (0, 0) $ of order $ N $ reduces to a singular point of $ \widetilde{E} $ given by $ y^2 + uxy = x^3 $ for some unit $ u \in k_v^\times $, so $ P $ is non-zero in $ E(K_v) / E_0(K_v) $, and hence $ N \mid c_v(E) $.

Now let $ N = 5 $. A change of variables translating $ P $ to $ (0, 0) $ transforms $ E $ to the Weierstrass equation in the Tate normal form $ \TT(f, f) $ for some $ f \in K \setminus k $ \cite[Section 2]{Sut12}. There are distinct places $ v_1 $ and $ v_2 $ such that $ v_1(f) > 0 $ and $ v_2(f) < 0 $. At $ v_1 $,
$$ v_1(a_1) = 0, \qquad v_1(a_2) = v_1(a_3) = v_1(f) > 0. $$
Now the change of variables $ (x, y) \mapsto (x / f^2, y / f^3) $ preserves $ P = (0, 0) $ and transforms $ E $ to the Weierstrass equation
$$ y^2 - \dfrac{f - 1}{f}xy - \dfrac{1}{f^2}y = x^3 - \dfrac{1}{f}x^2. $$
At $ v_2 $, Lemma \ref{lem:negative} gives
$$ v_2(a_1) = 0, \qquad v_2(a_2) = -v_2(f) > 0, \qquad v_2(a_3) = -2v_2(f) > 0. $$
The same argument with Lemma \ref{lem:odd} concludes.
\end{proof}

\renewcommand{\bibliofont}{\scriptsize}
\bibliographystyle{plain}
\bibliography{main}

@article{Ang25,
  author = {David Kurniadi Angdinata},
  title = {{On $ L $-values of elliptic curves twisted by cubic Dirichlet characters}},
  journal = {Canad. J. Math.},
  year = {2025},
  pages = {1--25},
  issn = {1496-4279},
  doi = {10.4153/S0008414X25101065},
  url = {https://doi.org/10.4153/S0008414X25101065},
}

@article{BCP97,
  author = {Wieb Bosma and John Cannon and Catherine Playoust},
  title = {{The Magma algebra system. I. The user language}},
  note = {Computational algebra and number theory (London, 1993)},
  journal = {J. Symbolic Comput.},
  volume = {24},
  year = {1997},
  number = {3-4},
  pages = {235--265},
  issn = {0747-7171},
  doi = {10.1006/jsco.1996.0125},
  url = {http://dx.doi.org/10.1006/jsco.1996.0125},
}

@article{BHPPPSSU20,
  author = {Lisa Berger and Chris Hall and Ren\'e Pannekoek and Jennifer Park and Rachel Pries and Shahed Sharif and Alice Silverberg and Douglas Ulmer},
  title = {{Explicit arithmetic of Jacobians of generalized Legendre curves over global function fields}},
  journal = {Mem. Amer. Math. Soc.},
  volume = {266},
  year = {2020},
  number = {1295},
  pages = {v+131},
  issn = {0065-9266,1947-6221},
  isbn = {978-1-4704-4219-4; 978-1-4704-6253-6},
  doi = {10.1090/memo/1295},
  url = {https://doi.org/10.1090/memo/1295},
}

@article{BR22,
  author = {Alexander Barrios and Manami Roy},
  title = {{Local data of rational elliptic curves with nontrivial torsion}},
  journal = {Pacific J. Math.},
  volume = {318},
  year = {2022},
  number = {1},
  pages = {1--42},
  issn = {0030-8730,1945-5844},
  doi = {10.2140/pjm.2022.318.1},
  url = {https://doi.org/10.2140/pjm.2022.318.1},
}

@article{CP80,
  author = {David Cox and Walter Parry},
  title = {{Torsion in elliptic curves over $ k(t) $}},
  journal = {Compositio Math.},
  volume = {41},
  year = {1980},
  number = {3},
  pages = {337--354},
  issn = {0010-437X,1570-5846},
  url = {https://www.numdam.org/item/CM_1980__41_3_337_0/},
}

@incollection{Gro11,
  author = {Benedict Gross},
  title = {{Lectures on the conjecture of Birch and Swinnerton-Dyer}},
  booktitle = {{Arithmetic of $ L $-functions}},
  series = {IAS/Park City Math. Ser.},
  volume = {18},
  pages = {169--209},
  publisher = {Amer. Math. Soc., Providence, RI},
  year = {2011},
  isbn = {978-0-8218-5320-7},
  doi = {10.1090/pcms/018/08},
  url = {https://doi.org/10.1090/pcms/018/08},
}

@phdthesis{Kru13,
  author = {David Krumm},
  title = {{Quadratic Points on Modular Curves}},
  school = {University of Georgia},
  year = {2013},
  pages = {161},
  isbn = {978-0438-99004-3},
  url = {https://openscholar.uga.edu/record/18468},
}

@article{Lor11,
  author = {Dino Lorenzini},
  title = {{Torsion and Tamagawa numbers}},
  journal = {Ann. Inst. Fourier (Grenoble)},
  volume = {61},
  year = {2011},
  number = {5},
  pages = {1995--2037 (2012)},
  issn = {0373-0956,1777-5310},
  doi = {10.5802/aif.2664},
  url = {https://doi.org/10.5802/aif.2664},
}

@article{Lor25,
  author = {Dino Lorenzini},
  title = {{Torsion and exceptional units}},
  journal = {Acta Arith.},
  volume = {217},
  year = {2025},
  number = {1},
  pages = {19--66},
  issn = {0065-1036,1730-6264},
  doi = {10.4064/aa231009-24-6},
  url = {https://doi.org/10.4064/aa231009-24-6},
}

@article{McD18,
  author = {Robert McDonald},
  title = {{Torsion subgroups of elliptic curves over function fields of genus 0}},
  journal = {J. Number Theory},
  volume = {193},
  year = {2018},
  pages = {395--423},
  issn = {0022-314X,1096-1658},
  doi = {10.1016/j.jnt.2018.05.017},
  url = {https://doi.org/10.1016/j.jnt.2018.05.017},
}

@article{Mel22,
  author = {Mentzelos Melistas},
  title = {{Tamagawa numbers of elliptic curves with torsion points}},
  journal = {Arch. Math. (Basel)},
  volume = {119},
  year = {2022},
  number = {2},
  pages = {155--165},
  issn = {0003-889X,1420-8938},
  doi = {10.1007/s00013-022-01722-4},
  url = {https://doi.org/10.1007/s00013-022-01722-4},
}

@misc{Mel25,
  author = {Mentzelos Melistas},
  title = {{Small Tamagawa numbers of elliptic curves with isogenies or torsion}},
  year = {2025},
  eprint = {2505.20479},
  archiveprefix = {arXiv},
  primaryclass = {math.NT},
  note = {\url{https://arxiv.org/abs/2505.20479}},
}

@article{Naj17,
  author = {Filip Najman},
  title = {{Tamagawa numbers of elliptic curves with $ C_{13} $ torsion over quadratic fields}},
  journal = {Proc. Amer. Math. Soc.},
  volume = {145},
  year = {2017},
  number = {9},
  pages = {3747--3753},
  issn = {0002-9939,1088-6826},
  doi = {10.1090/proc/13553},
  url = {https://doi.org/10.1090/proc/13553},
}

@manual{SageMath,
  author = {William Stein and others},
  organization = {The Sage Development Team},
  title = {{Sage Mathematics Software}},
  note = {\url{http://www.sagemath.org}},
}

@book{Sil09,
  author = {Joseph Silverman},
  title = {{The arithmetic of elliptic curves}},
  series = {Graduate Texts in Mathematics},
  volume = {106},
  edition = {Second},
  publisher = {Springer, Dordrecht},
  year = {2009},
  pages = {xx+513},
  isbn = {978-0-387-09493-9},
  doi = {10.1007/978-0-387-09494-6},
  url = {https://doi.org/10.1007/978-0-387-09494-6},
}

@book{Sil94,
  author = {Joseph Silverman},
  title = {{Advanced topics in the arithmetic of elliptic curves}},
  series = {Graduate Texts in Mathematics},
  volume = {151},
  publisher = {Springer-Verlag, New York},
  year = {1994},
  pages = {xiv+525},
  isbn = {0-387-94328-5},
  doi = {10.1007/978-1-4612-0851-8},
  url = {https://doi.org/10.1007/978-1-4612-0851-8},
}

@misc{Stacks,
  author = {The Stacks project authors},
  title = {{The Stacks project}},
  note = {\url{https://stacks.math.columbia.edu}},
}

@misc{Sut,
  author = {Andrew Sutherland},
  title = {{Raw equations for $ X_1(N) $}},
  note = {\url{https://math.mit.edu/~drew/X1_rawcurves.html}},
}

@article{Sut12,
  author = {Andrew Sutherland},
  title = {{Constructing elliptic curves over finite fields with prescribed torsion}},
  journal = {Math. Comp.},
  volume = {81},
  year = {2012},
  number = {278},
  pages = {1131--1147},
  issn = {0025-5718,1088-6842},
  doi = {10.1090/S0025-5718-2011-02538-X},
  url = {https://doi.org/10.1090/S0025-5718-2011-02538-X},
}

@phdthesis{Szy99,
  author = {Michael Szydlo},
  title = {{Flat regular models of elliptic schemes}},
  school = {Harvard University},
  year = {1999},
  pages = {195},
  isbn = {978-0599-20754-7},
  url = {https://www.szydlo.com/mixchar.pdf},
}

@inproceedings{Tat75,
  author = {John Tate},
  title = {{Algorithm for determining the type of a singular fiber in an elliptic pencil}},
  booktitle = {{Modular functions of one variable, IV (Proc. Internat. Summer School, Univ. Antwerp, Antwerp, 1972)}},
  series = {Lecture Notes in Math., Vol. 476},
  pages = {33--52},
  publisher = {Springer, Berlin-New York},
  year = {1975},
}

@article{Trb22,
  author = {Antonela Trbovi\'c},
  title = {{Tamagawa numbers of elliptic curves with prescribed torsion subgroup or isogeny}},
  journal = {J. Number Theory},
  volume = {234},
  year = {2022},
  pages = {74--94},
  issn = {0022-314X,1096-1658},
  doi = {10.1016/j.jnt.2021.09.007},
  url = {https://doi.org/10.1016/j.jnt.2021.09.007},
}

@incollection{Ulm11,
  author = {Douglas Ulmer},
  title = {{Elliptic curves over function fields}},
  booktitle = {{Arithmetic of $ L $-functions}},
  series = {IAS/Park City Math. Ser.},
  volume = {18},
  pages = {211--280},
  publisher = {Amer. Math. Soc., Providence, RI},
  year = {2011},
  isbn = {978-0-8218-5320-7},
  doi = {10.1090/pcms/018/09},
  url = {https://doi.org/10.1090/pcms/018/09},
}

@article{Ulm14,
  author = {Douglas Ulmer},
  title = {{Explicit points on the Legendre curve}},
  journal = {J. Number Theory},
  volume = {136},
  year = {2014},
  pages = {165--194},
  issn = {0022-314X,1096-1658},
  doi = {10.1016/j.jnt.2013.09.010},
  url = {https://doi.org/10.1016/j.jnt.2013.09.010},
}

@article{Ulm19,
  author = {Douglas Ulmer},
  title = {{On the Brauer-Siegel ratio for abelian varieties over function fields}},
  journal = {Algebra Number Theory},
  volume = {13},
  year = {2019},
  number = {5},
  pages = {1069--1120},
  issn = {1937-0652,1944-7833},
  doi = {10.2140/ant.2019.13.1069},
  url = {https://doi.org/10.2140/ant.2019.13.1069},
}

\end{document}